\documentclass{amsproc}
\newtheorem{theorem}{Theorem}[section]
\newtheorem{lemma}[theorem]{Lemma}
\newtheorem{corollary}[theorem]{Corollary}
\theoremstyle{definition}
\newtheorem{definition}[theorem]{Definition}
\newtheorem{example}[theorem]{Example}

\theoremstyle{remark}
\newtheorem{remark}[theorem]{Remark}
\numberwithin{equation}{section}

\usepackage{graphicx}
\usepackage{amssymb}
\usepackage{epstopdf}
\usepackage{epsfig}
\usepackage{psfrag}
\begin{document}

 \title{The Geometry of $p$-Adic Fractal Strings: A Comparative Survey\footnote{This article is in the Proceeding of the 11th International Conference on \textit{$p$-Adic Functional Analysis}
(Clermont-Ferrand, France, July 2010),
J.\ Araujo, B.\ Diarra and A.\ Escassut, eds., Contemporary Mathematics, American  Mathematical Society, Providence, R.\ I., 2011.}
 } 
\author{Michel L. Lapidus}
\address{Department of Mathematics, University of California, Riverside, CA 92521-0135}
\email{lapidus@math.ucr.edu}
\thanks{The research of the first author (MLL) was partially supported by the US National Science Foundation under the grant DMS-0707524.}

\author {L\~u' H\`ung}
\address{Department of Math \& CS, Hawai`i Pacific University, Honolulu, HI 96813-2785}
\email{hlu@hpu.edu}
\thanks{The research of the second author (LH) was partially supported by the Trustees' Scholarly Endeavor Program at Hawai`i Pacific University.}
\subjclass[2010]{Primary 11M41, 26E30, 28A12, 32P05, 37P20; Secondary 11M06, 11K41, 30G06, 46S10, 47S10, 81Q65.}

\keywords{Fractal geometry, $p$-adic analysis,   $p$-adic fractal strings, zeta functions, complex dimensions, Minkowski dimension,  tubes formulas, $p$-adic self-similar strings, lattice strings, 
Cantor, Euler and Fibonacci strings, nonarchimedean harmonic and functional analysis. }

\begin{abstract}
We give a brief overview  of the theory of complex dimensions of real (archimedean) fractal strings via an illustrative example, the ordinary Cantor string, and a detailed survey of the theory of $p$-adic (nonarchimedean) fractal strings and their complex dimensions. 
Moreover, we present an explicit volume formula for the tubular neighborhood of a $p$-adic fractal string 
$\mathcal{L}_p$, expressed in terms of the underlying complex dimensions.  Special attention will be focused on $p$-adic self-similar strings, in which the nonarchimedean theory takes a more natural form than its archimedean counterpart.  In   contrast with the archimedean setting, all $p$-adic self-similar strings are lattice and hence, their complex dimensions (as well as their zeros) are periodically distributed along finitely many vertical lines. The general theory is illustrated by some simple examples, the nonarchimedean Cantor, Euler, and Fibonacci strings.  Throughout this comparative survey of the archimedean and nonarchimedean theories of fractal (and possibly, self-similar) strings, we discuss analogies and differences between the real and $p$-adic situations. We close this paper by proposing several directions for future research, including seemingly new and challenging problems in $p$-adic (or rather, nonarchimedean) harmonic and functional analysis, as well as spectral theory. 
\end{abstract}

\maketitle

\tableofcontents



\begin{quote}
{\em
Nature is an infinite sphere of which the center is everywhere and the circumference nowhere.}
\hspace{\stretch{1}} Blaise Pascal 
\end{quote}


\section{Introduction} 

 In this survey, we present aspects of a geometric theory of $p$-adic (or nonarchimedean)  fractal strings, that is, bounded open subsets of the $p$-adic line $\mathbb Q_p$ having  a fractal subset of 
$\mathbb Q_p$ for ``boundary''.  This theory, developed by Michel Lapidus and 
L\~u\hspace{-.25em}\raisebox{-.1 em}{'} H\`ung  in [\ref{LapLu1}, \ref{LapLu2}], as well as by those same authors and Machiel van Frankenhuijsen in [\ref{LapLu3}],  extends in a natural way the theory of real (or archimedean)  fractal strings and their complex dimensions developed in  [\ref{L-vF1}, \ref{L-vF2}], and building on
 [\ref{Lap3}, \ref{LapMa}, \ref{LapPo}], for example. Following [\ref{LapLu1}, \ref{LapLu2}, \ref{LapLu3}],  we introduce suitable geometric zeta functions, the poles of which play the role for  $p$-adic fractal strings of the complex dimensions for the standard real fractal strings. Furthermore, we discuss the analogies and the differences between the real and $p$-adic fractal strings. 

More specifically, we recall the definition of $p$-adic self-similar strings introduced in [\ref{LapLu2}]; furthermore, we show (as in [\ref{LapLu2}]) that all $p$-adic self-similar strings are lattice (in a strong sense) and deduce from this fact the simple periodic structure of their complex dimensions. We also discuss the explicit fractal tube formulas obtained in [\ref{LapLu3}], both in the general case of (languid) $p$-adic fractal strings and that of $p$-adic self-similar strings. Throughout this paper, these various results are illustrated in the case of suitable nonarchimedean analogs of the Cantor and the Fibonacci strings (which are both self-similar), as well as in the case of a new (and non self-similar) $p$-adic fractal string, namely, the $p$-adic Euler string introduced in [\ref{LapLu3}]. 
Some particular attention is devoted to the nonarchimedean (or 3-adic) Cantor string (introduced and studied in [\ref{LapLu1}]), an appropriate counterpart of the archimedean Cantor string, whose `metric' boundary is the nonarchimedean (or 3-adic) Cantor set ([\ref{LapLu1}]), a suitable $p$-adic analog of the classic ternary Cantor set. 

We note that $p$-adic (or nonarchimedean) analysis has been used in various areas of mathematics (such as representation theory, number theory and arithmetic geometry), as well as (more speculatively) of mathematical and theoretical physics (such as quantum mechanics, relativity theory, quantum field theory, statistical and condensed matter physics, string theory and cosmology); see, e.g., [\ref{Drag}, \ref{Dykkv}, \ref{Ulam}, \ref{RTV}, \ref{VVZ}] and the relevant references therein.  In particular, it is believed by some authors that $p$-adic numbers (or, more generally, nonarchimedean fields) can be used to describe the geometry of spacetime at very high energies and hence, very small scales (i.e., below the Planck or the string scale); see, e.g., [\ref{Vol}].  Furthermore, several physicists and mathematical physicists have suggested that the small scale structure of spacetime may be fractal; see, e.g., [\ref{GibHaw}, \ref{HawIs}, \ref{Lap2}, \ref{Not}, \ref{WheFo}].

On the other hand, in the recent book [\ref{Lap2}], it has been suggested that fractal strings and their quantization, fractal membranes, may be related to aspects of string theory and that $p$-adic (and possibly, ad\`elic) analogs of these notions would be useful in this context in order to better understand the underlying (noncommutative) spacetimes and their moduli spaces ([\ref{Lap2}, \ref{LapNe}]).  The theory of $p$-adic fractal strings, once suitably `quantized',  may be helpful in further developing some of these ideas    and eventually providing a framework for unifying the real and $p$-adic fractal strings and membranes.  

The rest of this paper is organized as follows:

In \S\ref{CSS}, we give a brief survey of the main pertinent properties of the Cantor string and set, both in the real (or archimedean) and the 3-adic (or nonarchimedean) situations. This will serve, in particular, as a pedagogical introduction to the general theory of real and $p$-adic  fractal (and possibly self-similar) strings. 

In \S \ref{pfs}, we recall the definition of an arbitrary $p$-adic fractal string, along with some of the key pertaining notions (geometric zeta function and complex dimensions, as well as Minkowski dimension and content). We also discuss the  more technical question of how to suitably define and calculate the volume (i.e., $p$-adic Haar measure) of the `inner' $\varepsilon$-neighborhood (or inner tube) of a $p$-adic fractal string (\S\ref{inner tube}).\footnote{Most of the proofs given in this paper will be concentrated in \S\ref{inner tube} because they truly depend on the nonarchimedean (specifically, $p$-adic) nature of the underlying geometry.}
In \S\ref{Mdimension}, we then use these results to express this volume as an infinite sum over the underlying complex dimensions, thereby obtaining a nonarchimedean analog of the `fractal tube formula' of [\ref{L-vF1}, \ref{L-vF2}]. We illustrate this formula in \S\ref{explicit tf} by providing (as well as deriving via a direct computation) the fractal tube formula for the $p$-adic Euler string, the definition of which is given in \S\ref{Euler string}. 

In \S\ref{sss}, we focus on the important special class of $p$-adic self-similar strings, of which the 3-adic Cantor string and the 2-adic Fibonacci string are among the simplest examples. After having explained their construction in \S\ref{construction} via an iterated function (or self-similar) system, we provide 
(in \S\ref{gzf sss}--\ref{zeros and poles}) a detailed study of their geometric zeta functions and complex dimensions. It turns out that due to the discreteness of the valuation group of $\mathbb Q_p,$ all $p$-adic self-similar strings are `strongly lattice' (\S\ref{Periodicity}), from which it follows that their complex dimensions (along with the zeros of their geometric zeta functions) are periodically distributed along finitely 
many vertical lines (\S\ref{zeros and poles}).\footnote{This is not necessarily the case for a general `lattice' archimedean self-similar string. Moreover, a generic archimedean self-similar string is `nonlattice'. It follows that  the theory of $p$-adic self-similar fractal strings is more natural as well as simpler than its archimedean counterpart.} 
It follows that (under mild assumptions) the fractal tube formula of a $p$-adic self-similar string involves finitely many (multiplicative) periodic functions, one for each `line' of complex dimensions. We describe such explicit tube formulas in some detail in \S\ref{exact tf} and also obtain in \S\ref{amc} an explicit expression for the average Minkowski content of a $p$-adic self-similar string (and the associated nonarchimedean self-similar set).

Throughout this expository paper and comparative survey, we illustrate some of the main results by discussing the examples of the nonarchimedean Cantor, Euler and Fibonacci strings. Moreover, we point out the main analogies and differences between the archimedean and nonarchimedean theories of fractal strings. 

Finally, in \S\ref{conclusion}, we conclude this paper by proposing several possible research directions for future work in this new field. This includes, in particular, a possible extension of the nonarchimedean theory of fractal strings and their tube formulas to Berkovich spaces, along with seemingly new and quite challenging problems in nonarchimedean spectral, harmonic and functional analysis.

\section{Archimedean vs. Nonarchimedean Cantor Set and String}\label{CSS}

In this section, we briefly recall, for a simple but important example, some of the main notions pertaining to the  theory of real and $p$-adic fractal strings. Namely, the geometric zeta function, the complex dimensions, and the `fractal tube formula' that expresses the volume of the inner $\varepsilon$-neigborhoods of a suitable boundary of the string as a `fractal power series' with exponents involving the underlying complex dimensions. 

In the archimedean case, the classic example is the real Cantor string (\S\ref{CS}), whose boundary (or associated self-similar set) is the ternary Cantor set (\S\ref{C}); cf. [\ref{L-vF2}, Ch. 1].
The nonarchimedean counterpart of the real Cantor set or string is the 3-adic Cantor set or string, discussed in \S\ref{C3} or \S\ref{CS3}, respectively (and both introduced in [\ref{LapLu1}]). In particular, the 3-adic Cantor set is the nonarchimedean self-similar set naturally associated with the 3-adic Cantor string (which is a special case of $p$-adic  self-similar string, in the sense of [\ref{LapLu2}] and \S\ref{sss} below).

Finally, a comparative study of the archimedean and nonarchimedean Cantor strings and their respective fractal tube formulas is provided in \S\ref{comparison}. It will help us preview some of the main analogies and differences between the real and $p$-adic theory of fractal (and possibly, self-similar) strings, as further discussed and developed in \S\ref{pfs} and \S\ref{sss}.

\subsection{Archimedean (or Ternary) Cantor Set}\label{C}
The classical archimedean (or ternary) Cantor  set, denoted by $\mathcal{C}$, is the set that remains after iteratively removing the open middle third subinterval(s) from the closed unit interval $C_0=[0,1].$ 
The construction is illustrated in Figure 1. There, for each $n\geq 0,$ $C_n$ is the compact set defined as the union of $2^n$ compact intervals of length $3^{-n}$ and endpoints the ternary points of `scale' $n$ (i.e., of the form $\frac{3k+j}{3^n},$ with $k\in \mathbb N$ and $j=1,2$).
Hence, the archimedean Cantor  set $\mathcal{C}$ is equal to $\bigcap_{n=0}^{\infty} C_n$. 
\begin{figure}[h!]
\psfrag{dots}{$\vdots$}
\psfrag{C0}{$C_0$}
\psfrag{C1}{$C_1$}
\psfrag{C2}{$C_2$}
\psfrag{Cn}{$C_n$}
\raisebox{-1cm}
{\psfig{figure=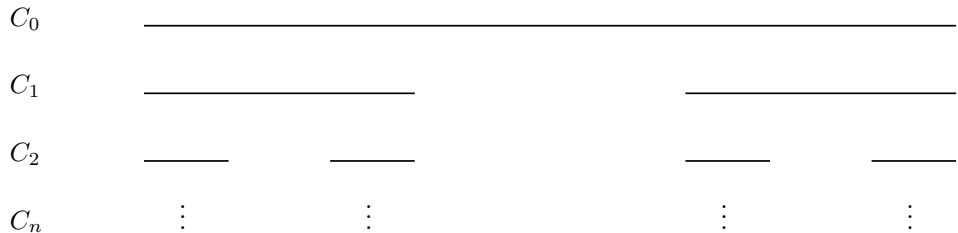, height=3.25 cm}}
\caption{Construction of the archimedean Cantor set $\mathcal{C}=\bigcap_{n=0}^{\infty} C_n$.}
\label{Cantor construction}
\end{figure}

For comparison with our results in the nonarchimedean case, we state 
the following well-known results (see, e.g., [\ref{Fal}, Ch. 9] and [\ref{GJ}, p. 50]):

\begin{theorem}\label{fixed point}
The archimedean Cantor  set $\mathcal C$ is self-similar.  More specifically, it is the unique nonempty, compact invariant set  in $[0,1]\subset \mathbb R$ generated by the iterated function system (IFS) $\mathbf{\Phi}=\{\Phi_1, \Phi_2\}$ of affine similarity contraction mappings of $[0, 1]$ into itself, where
\[\Phi_1(x)=\frac{x}{3} \quad \mbox{and} \quad  \Phi_2(x)= \frac{x}{3} + \frac{2}{3}.\]
That is, 
\[\mathcal{C}= \Phi_1(\mathcal{C})\cup \Phi_2(\mathcal{C}).\]
\end{theorem}

\begin{theorem}\label{expansion}
The archimedean Cantor set is characterized by the ternary expansion of its elements as 
\[
\mathcal{C}=\left\{ \tau \in [0,1] ~ :~\tau=a_0 + a_13^{-1} + a_23^{-2} + \cdots, ~a_j \in \{0,2\},
\forall j \geq 0\right\}.
\]
\end{theorem}

We note that, as usual, we choose the nonrepeating ternary expansion here
(so that none of the coefficient $a_j$ is equal to 1, and hence, the sequence of digits does not end with $\bar 1$, where the overbar indicates that 1 is repeated ad infinitum).
Such a precaution will not be needed in \S\ref{C3} for the elements of $\mathbb Q_3$  because the 3-adic expansion is unique.

\subsection{Archimedean (or Real) Cantor String}\label{CS}
The ordinary  \emph{archimedean} (or \emph{real}) \emph{Cantor string} $\mathcal{CS}$ is defined as the complement of the ternary Cantor  set in the closed unit interval $[0,1]$. By construction, the topological boundary of 
$\mathcal{CS}$ is the ternary Cantor  set $\mathcal{C}$. 
The Cantor string is one of the simplest and most important examples in the research monographs 
[\ref{L-vF1}, \ref{L-vF2}] by Lapidus and van Frankenhuijsen.  Indeed, it is used throughout those books to illustrate and motivate the general theory; see also, e.g., [\ref{Lap}] and [\ref{LapPo}].  From the point of view of the theory of fractal strings and their complex dimensions 
[\ref{L-vF1}, \ref{L-vF2}], it suffices to consider the sequence 
$\{l_n\}_{n\in \mathbb N^*}$ of lengths associated to $\mathcal{CS}$.\footnote{Here and thereafter, we let $\mathbb N:=\{0, 1, 2, \ldots\}$ and $\mathbb N^*:=\{1, 2, 3, \ldots\}$.} More specifically,  these are the distinct lengths of the intervals of which the bounded open set $\mathcal {CS} \subset \mathbb R$ is composed, counted according to their multiplicities.\footnote{For an arbitrary real or $p$-adic fractal string, there are two equivalent ways of keeping track of its associated lengths. Either by considering its distinct lengths, counted according to their multiplicities, as above, or else by considering the sequence of all of its lengths, written in nonincreasing order (and hence, tending to zero, except in the trivial case when the fractal string is composed of finitely many intervals).}
 Accordingly, the archimedean Cantor string consists of $m_1=1$ interval of length $l_1=1/3$, $m_2=2$ intervals of length $l_2=1/9$, $m_3=4$ intervals of length $l_3=1/27$, and so on;  
see Figure 2.

\begin{figure}[tb]\label{Cantor string}
\psfrag{dots}{$\dots$}
\psfrag{l1}{$l_1$}
\psfrag{l2}{$l_2$}
\psfrag{l3}{$l_3$}
\psfrag{ln}{$l_n$}
\raisebox{-1cm}
{\psfig{figure=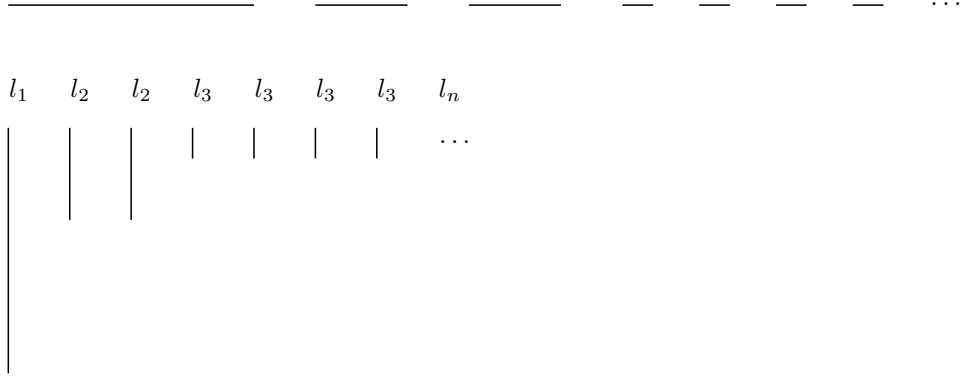, height=5.4cm}}
\caption{The archimedean Cantor string $\mathcal {CS}$ (above); the Cantor string viewed as a fractal harp (below).}
\end{figure}

Important information about the geometry of $\mathcal{CS}$, e.g., the Minkowski dimension and the Minkowski measurability, is contained in its  \emph{geometric zeta function}
\begin{equation}\label{zetaCS}
\zeta_{\mathcal{CS}}(s):=\sum_{n=1}^{\infty}m_n\cdot l_n^s=
\sum_{n=1}^{\infty}\frac{2^{n-1}}{ 3^{ns}} 
=\frac{3^{-s}}{1-2\cdot 3^{-s}} \hspace{0.5cm} \mbox{for}~ \Re(s)>D,
\end{equation}
where $D=\log2/\log3=\log_32$ is the \emph{Minkowski dimension} of the ternary Cantor  set.\footnote{Throughout this paper, $\log t$ denotes the natural logarithm of $t>0$.} 
In addition, 
$\zeta_{\mathcal{CS}}$ can be extended to a meromorphic function on the entire complex plane
$\mathbb{C}$, as given by the last expression in (\ref{zetaCS}).  The corresponding set of poles of $\zeta_{\mathcal{CS}}$ is then given by
\begin{equation}\label{cdCS}
 \mathcal{D}_{\mathcal {CS}}=\{   D+i \nu \mathbf{p}~|~\nu \in \mathbb{Z} \},
 \end{equation}
where $i:=\sqrt{-1}$ and $\mathbf{p}:=2\pi / \log3$ is the \emph{oscillatory period} of $\mathcal {CS}$.
The set $\mathcal{D}_{\mathcal {CS}}$ is called the set of \emph{complex dimensions} of the real Cantor string; see Figure 5 in \S\ref{CS3}.  

For $\varepsilon >0$, let $V_{\mathcal {CS}}(\varepsilon)$ be the volume of the inner tubular neighborhood of the boundary of the real Cantor string, i.e., 
$\partial(\mathcal {CS})=\mathcal C,$ with radius $\varepsilon$:
\begin{equation}\label{V_CS}
V_{\mathcal {CS}}(\varepsilon)=\mu_{L}(\{x\in \mathcal {CS}~|~d(x,\mathcal {C})<\varepsilon \}),
\end{equation}
where $\mu_L$ is the one-dimensional Lebesgue measure on $\mathbb R$. Then it can be computed directly (as in [\ref{L-vF2}, \S1.12]) to depend only on the lengths of $\mathcal{CS} $ and to be given by
\begin{equation}\label{tube formula of CS}
V_{\mathcal{CS}}(\varepsilon)=\frac{1}{2\log3}\sum_{\omega\in\mathcal{D}_{\mathcal{CS}}}\frac{(2\varepsilon)^{1-\omega}}{\omega(1-\omega)}-2\varepsilon,
\end{equation}
where $\mathcal{D}_{\mathcal{CS}}$ is as in (\ref{cdCS}).

The general theme of the monographs [\ref{L-vF1}, \ref{L-vF2}] is that \emph{the complex dimensions describe oscillations in the geometry and the spectrum of a fractal string.} In particular, due to the presence of nonreal complex dimensions on the vertical line $\Re(s)=D,$ there are oscillations of order $D$ in the geometry of $\mathcal{CS}$ and therefore its boundary, the ternary Cantor set, is \emph{not} Minkowski measurable; see [\ref{LapPo}], [\ref{L-vF2}, \S 1.1.2]. 

\subsection{Nonarchimedean (or 3-Adic) Cantor Set}\label{C3}

Our goal in this section is to provide a natural  nonarchimedean (or $p$-adic) analog of the classic ternary Cantor  set 
$\mathcal C$ and to show that it satisfies a counterpart of many of the key properties of $\mathcal C$ in this nonarchimedean context.  Furthermore, we will show in \S\ref{CS3} that the corresponding $p$-adic fractal string, called the 
nonarchimedean (or 3-adic)  Cantor string and denoted by $\mathcal {CS}_3$, is an exact analog of the ordinary archimedean Cantor string $\mathcal{CS}$, a central example in the theory of real  fractal strings and their complex dimensions [\ref{L-vF1}, \ref{L-vF2}]. The nonarchimedean Cantor set and string were both introduced and studied in detail in [\ref{LapLu1}].

We begin by recalling a few simple facts concerning the field of $p$-adic numbers $\mathbb Q_p,$ equipped with the standard $p$-adic absolute value $|\cdot|_p$ and associated topology; see, e.g., [\ref{Kob}, \ref{Rob}, \ref{Sch}].\footnote{Here and thereafter, $|\cdot|_p$ is normalized in the usual way; namely, 
$|p^k|_p=p^{-k},$ for any $k\in \mathbb Z$.} 
As is well known, 
every $z\in \mathbb Q_p$ has a unique representation as a convergent infinite series in ($\mathbb Q_p, |\cdot|_p$):
\[z=a_{v}p^{v} + \cdots + a_0 + a_1p + a_2 p^2 + \cdots, \]
for some $v\in \mathbb{Z}$ and  $a_j \in \{0, 1, \dots, p-1\}$ for all $j\geq v$.
An important subset of $\mathbb{Q}_p$ is the unit ball, $\mathbb{Z}_p=\{x\in \mathbb Q_p~:~ |x|_p\leq 1\}$, which can also be represented as follows:
\[\mathbb{Z}_p=\{ a_0 + a_1p + a_2 p^2 + \cdots ~|~a_j \in \{0, 1, \dots, p-1\},~ \forall j\geq 0 \}.
\] 
Using this \emph{p-adic expansion}, we can easily see that 
\begin{equation}\label{decomposition}
\mathbb Z_p=\bigcup_{c=0}^{p-1} (c+p\mathbb Z_p),
\end{equation}
where $c+p\mathbb Z_p=\{y\in \mathbb Q_p: |y-c|_p\le p^{-1}\}$ is the $p$-adic ball (or interval) of center $c$ and radius $p^{-1}$ .

A remarkable property of the nonarchimedean `unit interval' $\mathbb Z_p$ of $\mathbb Q_p$, which does not have any analog for the archimedean unit interval $[0,1]$ of $\mathbb R$, is that  $\mathbb Z_p$ is a ring, and in particular, is stable under addition; see, e.g., [\ref{Har}] for a thorough discussion of this point. Indeed, $|a+b|_p\leq \max\{|a|_p, |b|_p\}\leq 1,$ if $a,b\in \mathbb Z_p.$
Moreover, $\mathbb Z_p$ is a compact group and as such, admits a unique translation invariant measure, to be also denoted by $\mu_H$, which is the restriction to $\mathbb Z_p$ of Haar measure
 $\mu_H$ on $\mathbb Q_p$. 
Finally, note that unlike its real counterpart $[0,1]$ (or $[-1,1]$), $\mathbb Z_p$ is not connected; actually, it is totally disconnected. This well known fact, combined with the `self-duplicating property' 
(\ref{decomposition}), will naturally lead us to suitably modify many of the definitions and results of the standard theory of ordinary real fractal strings.

\begin{figure}[h]
\psfrag{dots}{$\vdots$}
\psfrag{C0}{$T_0$}
\psfrag{C1}{$T_1$}
\psfrag{C2}{$T_2$}
\psfrag{Cn}{$T_n$}
\psfrag{Z3}{$\mathbb Z_3$}
\psfrag{Z0}{$0+3\mathbb Z_3$}
\psfrag{Z1}{$1+3\mathbb Z_3$}
\psfrag{Z2}{$2+3\mathbb Z_3$}
\psfrag{1}{$0+9\mathbb Z_3$}
\psfrag{2}{$3+9\mathbb Z_3$}
\psfrag{3}{$6+9\mathbb Z_3$}
\psfrag{4}{$2+9\mathbb Z_3$}
\psfrag{5}{$5+9\mathbb Z_3$}
\psfrag{6}{$8+9\mathbb Z_3$}
\raisebox{-1cm}
{\psfig{figure=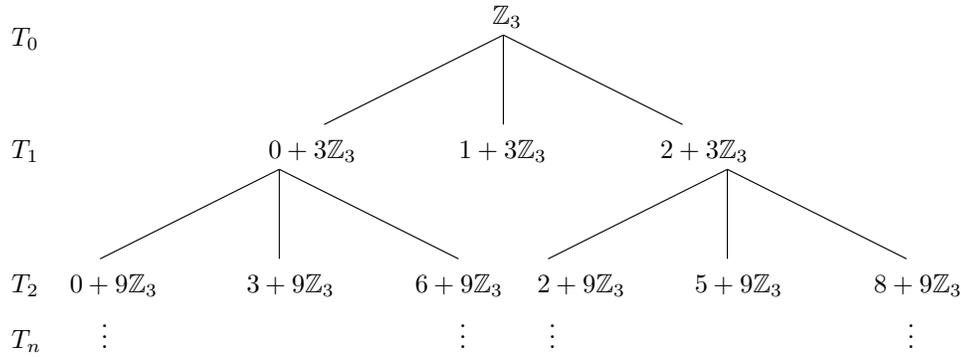, height=4.7cm}}
\caption{Construction of the nonarchimedean Cantor set $\mathcal{C}_3=\bigcap_{n=0}^{\infty} T_n.$}
\end{figure}

Consider the ring of 3-adic integers $\mathbb{Z}_3$. In a procedure reminiscent of the construction of the classic ternary Cantor set (see \S\ref{C}),  we construct the \emph{nonarchimedean} (or\emph{ 3-adic}) \emph{Cantor set} as follows.  First, we subdivide  $T_0=\mathbb{Z}_3$ into 3 equally long subintervals. We then remove the ``middle  third'' subinterval
$1+3\mathbb{Z}_3$ and call $T_1$ the remaining set: $T_1=0+3\mathbb{Z}_3 \cup 2+3\mathbb{Z}_3$.
We then repeat this process with each of the remaining subintervals, i.e., with $0+3\mathbb{Z}_3$ and  $2+3\mathbb{Z}_3$.  Finally, we define the nonarchimedean Cantor set $\mathcal{C}_3$ to be
$\bigcap_{n=0}^{\infty} T_n$; see Figure 3.  Here, for each $n\geq 0,$ the compact set $T_n$ is the union of $2^n$ $3$-adic intervals of scale $n$ (i.e., of radius or diameter $3^{-n}$). Note that $\mathcal C_3$ is compact, as the intersection of a decreasing sequence of compact subsets of $\mathbb Z_3.$

The nonarchimedean analog of Theorem \ref{fixed point} is then given  by Theorem \ref{invariant}: 
\begin{theorem}\label{invariant}
The nonarchimedean Cantor set $\mathcal{C}_3$ is self-similar.  More specifically, it is the unique nonempty, compact invariant set in $\mathbb Z_3 \subset \mathbb Q_3$ generated by the IFS
 $\mathbf{\Phi}=\{\Phi_1, \Phi_2\}$ of affine similarity contraction mappings of $\mathbb{Z}_3$ into itself, where
\begin{equation}\label{cantor maps}
\Phi_1(x)=3x \quad \mbox{and}\quad \Phi_2(x)=3x+2.
\end{equation}
That is, 
\[
\mathcal{C}_3= \Phi_1(\mathcal{C}_3)\cup \Phi_2(\mathcal{C}_3).
\]
\end{theorem}

 The next result is also a counterpart of a well known property of the ternary Cantor set (which we have omitted to recall in \S\ref{C}, by necessity of concision). It is a simple consequence of the self-similarity of the nonarchimedean Cantor set $\mathcal C_3$, as expressed by Theorem \ref{invariant}, and provides a useful introduction to the more general notion of  $p$-adic self-similar string to be discussed in 
 \S\ref{sss}. 
 
 \begin{theorem}\label{Cantor} 
 Let $W_{\alpha}=\{1,2\}^{\alpha}$ be the set of all finite words, on two symbols, of a given length $\alpha\geq 0$. Then 
 \[
 \mathcal{C}_3=\bigcap_{\alpha=0}^{\infty} \bigcup_{w\in W_{\alpha}} \Phi_w (\mathbb{Z}_3),
 \]
 where $\Phi_w:=\Phi_{w_{\alpha}} \circ \cdots  \circ \Phi_{w_1}$ for~ $w=(w_1, \dots, w_{\alpha})\in W_{\alpha}$ and the maps $\Phi_{w_j} $ are as in Equation~(\ref{cantor maps}). 
 \end{theorem}

\begin{figure}[h]
\psfrag{dots}{$\vdots$}
\psfrag{C0}{$C_0$}
\psfrag{C1}{$C_1$}
\psfrag{C2}{$C_2$}
\psfrag{Cn}{$C_n$}
\psfrag{Z3}{$\mathbb Z_3$}
\psfrag{Z0}{$\Phi_1(\mathbb Z_3)$}
\psfrag{Z1}{$1+3\mathbb Z_3$}
\psfrag{Z2}{$\Phi_2(\mathbb Z_3)$}
\psfrag{1}{$\Phi_{11}(\mathbb Z_3)$}
\psfrag{2}{$3+9\mathbb Z_3$}
\psfrag{3}{$\Phi_{21}(\mathbb Z_3)$}
\psfrag{4}{$\Phi_{12}(\mathbb Z_3)$}
\psfrag{5}{$5+9\mathbb Z_3$}
\psfrag{6}{$\Phi_{22}(\mathbb Z_3)$}
\raisebox{-1cm}
{\psfig{figure=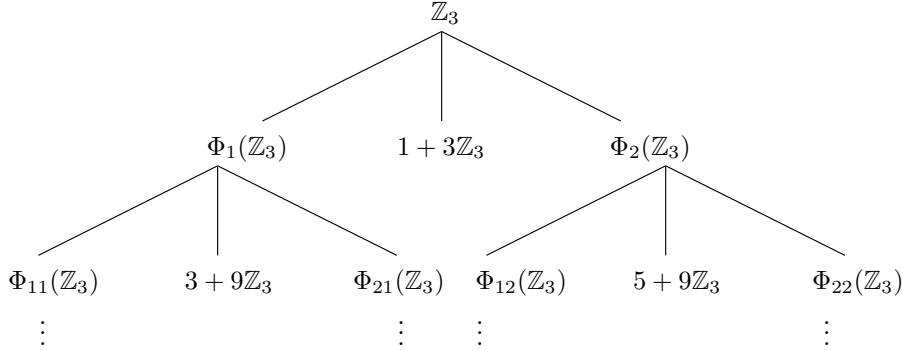, height=4.7cm}}
\caption{Construction of the nonarchimedean Cantor set $\mathcal{C}_3$ via an Iterated Function System (IFS).}
 \end{figure}

The following result is the nonarchimedean analog of Theorem \ref{expansion}:
\begin{theorem}\label{02}
The nonarchimedean Cantor set is characterized by the 3-adic expansion of its elements. That is,
  \[
\mathcal{C}_3=\left\{ \tau \in \mathbb Z_3~|~\tau=a_0 + a_1 3 +a_2 3^2 + \cdots,  a_j\in \{0,2\},~
\forall j\geq0\right\}.
\]
\end{theorem}

\begin{theorem}\label{homeomorphism}
The ternary Cantor  set $\mathcal{C}$ and the nonarchimedean Cantor set $\mathcal{C}_3$ are homeomorphic. 
\end{theorem}

\begin{proof} 
Let $\gamma:\mathcal{C}\rightarrow \mathcal{C}_3$ be the map sending 
\begin{equation}\label{rep}
\sum_{j=0}^{\infty}a_j 3^{-j} \mapsto \sum_{j=0}^{\infty}a_j 3^j,
\end{equation}
where $a_j\in \{0,2\}$ for all $j\geq0$. We note that on the left-hand side of (\ref{rep}), we use the ternary expansion in $\mathbb R$, whereas on the right-hand side we use the 3-adic expansion in $\mathbb Q_3$.
Then, in light of Theorems \ref{expansion} and \ref{02}, $\gamma$ is a continuous bijective map from $\mathcal{C}$ onto $\mathcal{C}_3$. 
Since both $\mathcal{C}$ and $\mathcal{C}_3$ are compact spaces in their respective natural metric topologies, $\gamma$ is a homeomorphism. 
 \end{proof}
 
 \begin{remark}
 In view of Theorem \ref{homeomorphism}, like its archimedean counterpart, the nonarchimedean Cantor set $\mathcal C_3$ is compact, totally disconnected, uncountably infinite and has no isolated points.  In particular, it is a perfect and complete metric space; furthermore, its topological dimension is 0.
 \end{remark}

\subsection{Nonarchimedean (or 3-Adic) Cantor String}\label{CS3}
The \emph{nonarchimedean} (or\emph{ 3-adic}) \emph{Cantor string} $\mathcal{CS}_3$ is defined to be 
\begin{equation}
\mathcal{CS}_3:=(1+3\mathbb{Z}_3) \cup  (3+9\mathbb{Z}_3) \cup  (5+9\mathbb{Z}_3) \cup \cdots
= \mathbb{Z}_3 \backslash \mathcal{C}_3,
\end{equation}
the complement of $\mathcal C_3$ in $\mathbb Z_3$; see the ``middle'' parts of Figure 4.
 Therefore, by analogy with the relationship between the archimedean Cantor set and Cantor string, the nonarchimedean Cantor set $\mathcal C_3$ can be thought of as a kind of ``boundary" of the nonarchimedean Cantor string.  Certainly, $\mathcal{C}_3$ is not the topological boundary of $\mathcal{CS}_3$ because the latter boundary is empty.  

As was alluded to earlier, since $\mathbb Q_p$ is a locally compact group, there is a unique translation invariant positive measure on $\mathbb Q_p$, called \emph{Haar measure} and denoted by  
$\mu_H,$   normalized so that $\mu_H(\mathbb Z_p)=1$ and hence
 $ \mu_H(a+p^k\mathbb{Z}_p)=p^{-k}$, for any $a\in \mathbb Q_p$ and $k\in \mathbb Z$; see 
 [\ref{Kob}], [\ref{Rob}], [\ref{Sch}]. As in the real case in \S\ref{CS}, we may identify $\mathcal{CS}_3$ with the sequence of lengths $l_n=3^{-n},$ counted with multiplicities $m_n=2^{n-1},$ for all $n\geq 1$.  
 
\begin{figure}[h]\label{cd}
\begin{center}
\psfrag{p}{$\mathbf{p}$}
\psfrag{10}{$10$}
\psfrag{0}{$0$}
\psfrag{1}{$1$}
\psfrag{D}{$D$}
\raisebox{-1cm}
{\psfig{figure=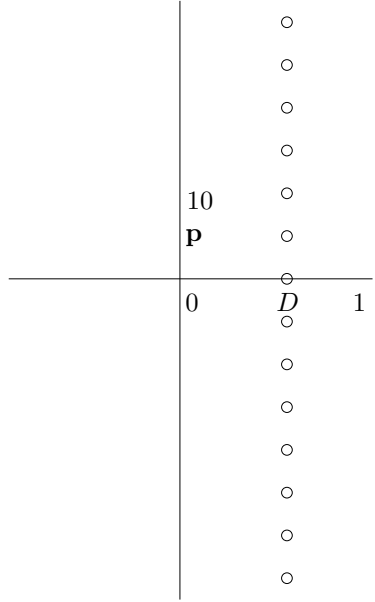, height=8cm}}
\caption{The set of complex dimensions, $\mathcal D_{\mathcal{CS}}=\mathcal D_{\mathcal{CS}_3},$
 of the archimedean and nonarchimedean Cantor strings, 
$\mathcal {CS}$ and $\mathcal{CS}_3$.}
\end{center}
\end{figure}

 Clearly, these lengths are given by the Haar measure of the $2^{n-1}$  3-adic intervals 
 $\{I_{n,q}\}_{q=1}^{2^{n-1}}$ of scale $n$ (and hence, of length $3^{-n}$) composing the level $n$ approximation to $\mathcal {CS}_3$. In the sequel, in agreement with the general definition of the geometric zeta function of a $p$-adic fractal string to be given in \S\ref{pfs}, 
 $\zeta_{\mathcal {CS}_3}(s)$ is initially defined by the following convergent Dirichlet series:
 \[
 \zeta_{\mathcal{CS}_3}(s):=\sum_{n=1}^{\infty}2^{n-1}(\mu_H(I_{n,q}))^s
 =\sum_{n=1}^{\infty}\frac{2^{n-1}}{3^{ns}},
  \]
  for $\Re(s)>\log_32$.

 The following theorem provides the exact analog of Equations (\ref{zetaCS}) and 
 (\ref{cdCS}):
 \begin{theorem}\label{zcd}
 The geometric zeta function of the nonarchimedean Cantor string is meromorphic in all of $\mathbb C$ and is given by
 \begin{equation}\label{zetaCS3}
\zeta_{\mathcal{CS}_3}(s)=\frac{3^{-s}}{1-2\cdot 3^{-s}}, \hspace{1cm} \mbox{for}~ s\in \mathbb C.
\end{equation}
Hence, the set of complex dimensions  of $\mathcal {CS}_3$ is given by
\begin{equation}\label{vcdCS3}
\mathcal{D}_{\mathcal {CS}_3}=\{   D+i \nu \mathbf{p}~|~\nu \in \mathbb{Z} \},
\end{equation}
where $D=\log_32$ is the dimension of $\mathcal {CS}_3$ and $\mathbf{p}=2\pi / \log3$ is its oscillatory period; see Figure 5.
 \end{theorem}

\begin{remark}
It is proved in [\ref{LapLu3}] that $D$ is the Minkowski dimension of $\mathcal {CS}_3 \subset \mathbb Z_3$; see Theorem \ref{D=sigma} below.  Clearly, it follows from the above computation that $D$ is also the abscissa of convergence of the Dirichlet series initially defining 
$\zeta_{\mathcal{CS}_3}.$ 
\end{remark}

\begin{remark}
We will see in Example \ref{cantor} that $\mathcal {CS}_3$ is a 3-adic \emph{self-similar}  string (in the sense of [\ref{LapLu2}] and \S\ref{sss}), with associated nonarchimedean self-similar set $\mathcal C_3$ and such that 
$\mathcal{CS}_3=\mathbb Z_3\backslash \mathcal C_3,$ the complement of the 3-adic Cantor set in $\mathbb Z_3.$
\end{remark}

The following result is the analog (for the nonarchimedean Cantor string $\mathcal {CS}_3$) of Theorem \ref{Cantor}. It provides a precise description of $\mathcal{CS}_3$ as a  countable disjoint union of 3-adic intervals. 
It also admits an archimedean counterpart, for the ternary Cantor string 
$\mathcal {CS}$ (which we have omitted to state in \S\ref{CS}). 
As we shall see in the more general context of \S\ref{sss},  the property described in that theorem follows from the self-similartiy of the nonarchimedean Cantor string $\mathcal {CS}_3$; see Figure 6.

 \begin{theorem}\label{G}
 With the same notation as in Theorem ~\ref{Cantor}, we have that
\[
 \mathcal{CS}_3=\bigcup_{\alpha=0}^{\infty} \bigcup_{w\in W_{\alpha}} \Phi_w (1+3 \mathbb{Z}_3).
 \]
 \end{theorem}

 \begin{figure}[h]
\psfrag{dots}{$\vdots$}
\psfrag{Z3}{$\mathbb Z_3$}
\psfrag{Z0}{$0+\mathbb Z_3$}
\psfrag{Z1}{$1+3\mathbb Z_3=G$}
\psfrag{Z2}{$2+\mathbb Z_3$}
\psfrag{1}{$0+9\mathbb Z_3$}
\psfrag{2}{$\Phi_1(G)$}
\psfrag{3}{$6+9\mathbb Z_3$}
\psfrag{4}{$2+9\mathbb Z_3$}
\psfrag{5}{$\Phi_2(G)$}
\psfrag{6}{$8+9\mathbb Z_3$}
\raisebox{-1cm}
{\psfig{figure=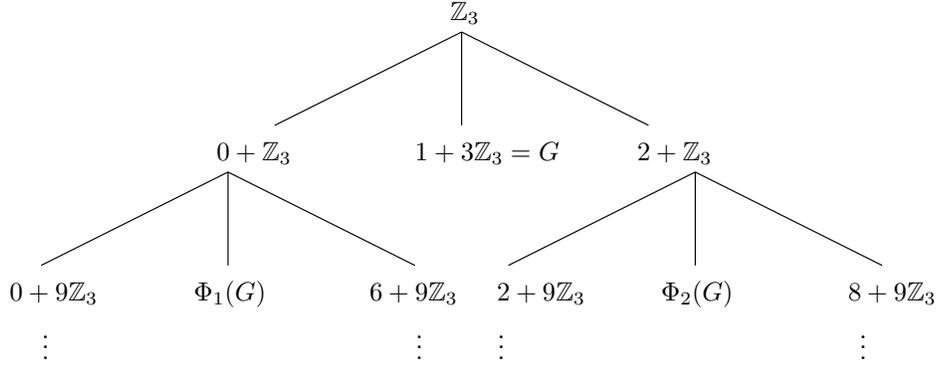, height=4.9cm}}
\caption{Construction of the nonarchimedean Cantor string $\mathcal {CS}_3$ via an IFS.}
\end{figure}

 The counterpart for the nonarchimedean Cantor string $\mathcal {CS}_3$ of the explicit tube formula (\ref{tube formula of CS}) for the archimedean Cantor string $\mathcal {CS}$ is given by
 \begin{equation}\label{tfCS3}
V_{\mathcal{CS}_3}(\varepsilon)=
\frac{1}{6\log3} \sum_{\omega \in \mathcal D_{\mathcal {CS}_3}}
\frac{\varepsilon^{1-\omega  }}{1-\omega},
\end{equation}
where $\mathcal D_{\mathcal {CS}_3}$ is given by (\ref{vcdCS3}).

\begin{remark}
With the exception of the fractal tube formula (\ref{tfCS3}) for $\mathcal {CS}_3$, which is derived in [\ref{LapLu3}], all of the results stated in \S\ref{CSS} are obtained in [\ref{LapLu1}], where the interested reader can find their detailed proofs. 
\end{remark}
\begin{remark}
The precise definition of the volume of the tubular neighborhood of a $p$-adic fractal string (and in particular, of the 3-adic Cantor string) will be given in Definition \ref{thin inner tube} (which makes use of  Definition \ref{volume definition}) of \S\ref{inner tube}. \end{remark}

\subsection{A Comparative Study of the  Real and 3-Adic Cantor Strings}\label{comparison}
A glance at Equations (\ref{cdCS}) and (\ref{vcdCS3}) shows that the archimedean and nonarchimedean Cantor strings $\mathcal {CS}$ and $\mathcal {CS}_3$ have the exact same set of complex dimensions:
\begin{equation}
\mathcal D_{\mathcal {CS}}=\mathcal D_{\mathcal {CS}_3}=\{   D+i \nu \mathbf{p}~|~\nu \in \mathbb{Z} \};\end{equation}
in particular, they have the same Minkowski dimension $D=\log_32$ and the same oscillatory period $\mathbf{p}=2\pi/\log3.$
In the present case, the complex dimension are in arithmetic progression (with period $\mathbf{p}$) along a single vertical line, 
$\Re(s)=D$, and they are simple (i.e., they are simple poles of the geometric zeta function).

We will see in \S\ref{Periodicity} and \S\ref{zeros and poles} that $p$-adic self-similar strings (of which $\mathcal {CS}_3$ is the simplest, nontrivial example) are always lattice, in a strong sense, which implies that their complex dimensions are periodically distributed along finitely many vertical lines, beginning with the rightmost line $\Re(s)=D,$ where $D$ is both the abscissa of convergence and the Minkowski dimension of the string. 

We next focus our attention on the fractal tube formulas for $\mathcal {CS}$ and $\mathcal {CS}_3$, as given by (\ref{tube formula of CS}) and (\ref{tfCS3}), respectively.\footnote{For now, we neglect the lower order term $2\varepsilon$ in (\ref{tube formula of CS}), which does not have a counterpart in (\ref{tfCS3})} 
In each case, one sums over the complex dimensions $\omega$ in 
$\mathcal D_{\mathcal{CS}}=\mathcal D_{\mathcal{CS}_3}$ a certain expression of $\varepsilon$ and 
$\omega$; namely, 
\[\frac{(2\varepsilon)^{1-\omega}}{(2\log3)\omega(1-\omega)} \quad \mbox{or}  \quad \frac{\varepsilon^{1-\omega  }}{(6\log3)(1-\omega)},\]
which can be interpreted as the residue at $s=\omega$ of the so-called `tubular zeta function' of   $\mathcal {CS}$ or $\mathcal {CS}_3$, respectively
(cf. [\ref{LapPe1}--\ref{LPW}] and Remark \ref{tzf} in \S\ref{explicit tf}).

Alternatively, since $\zeta_{\mathcal{CS}}(s)=\zeta_{\mathcal{CS}_3}(s),$ in light of (\ref{zetaCS}) and (\ref{zetaCS3}),\footnote{Here and thereafter, we denote by $res(f(s); \omega)$ (or $res(f; \omega)$, when no ambiguity may arise) the residue of a meromorphic function $f$ at the pole $s=\omega$.} 
\begin{equation}\label{residueCS}
res(\zeta_{\mathcal{CS}}; \omega)=res(\zeta_{\mathcal{CS}_3}; \omega)=\frac{1}{2\log3},
\end{equation}
for every $\omega\in \mathcal D_{\mathcal{CS}}=\mathcal D_{\mathcal{CS}_3}$, we can rewrite (\ref{tube formula of CS}) and (\ref{tfCS3}) as follows:
\begin{equation}\label{2.14}
V_{\mathcal{CS}}(\varepsilon)+2\varepsilon=\sum_{\omega\in\mathcal{D}_{\mathcal{CS}}}
res(\zeta_{\mathcal{CS}}; \omega)\frac{(2\varepsilon)^{1-\omega}}{\omega(1-\omega)}
\end{equation}
and
 \begin{equation}\label{2.15}
V_{\mathcal{CS}_3}(\varepsilon)=3^{-1}
\sum_{\omega \in \mathcal D_{\mathcal {CS}_3}}res(\zeta_{\mathcal{CS}_3}; \omega)
\frac{\varepsilon^{1-\omega  }}{1-\omega}
\end{equation}
(recall that for $\mathcal{CS}_3$, $p=3$ is the underlying prime).

Finally, we provide an additional reformulation of the tube formulas  (\ref{tube formula of CS}) and (\ref{tfCS3}). This reformulation will make apparent the role played by the complex dimensions. Namely, the \emph{real parts} of the complex dimensions govern the \emph{amplitudes} of the underlying \emph{oscillations}, while their \emph{imaginary parts} are directly linked with the \emph{frequencies} of these oscillations. More specifically, (\ref{tube formula of CS}) or (\ref{2.14}) implies that 
\begin{equation}\label{1}
(2\varepsilon)^{-(1-D)}(V_{\mathcal{CS}}(\varepsilon)+o(1))=G_{\mathcal{CS}}(\log_3\varepsilon^{-1}),
\end{equation}
where $o(1)\rightarrow 0$ as $\varepsilon \rightarrow 0^+$ and 
\begin{equation}\label{2}
G_{\mathcal{CS}}(x):= \frac{1}{2\log3}\sum_{n\in \mathbb Z}
\frac{e^{2\pi inx}}{(D+in\mathbf{p})(1-D-in\mathbf{p})}
\end{equation}
is a bounded, nonconstant periodic function of period 1 on $\mathbb R$.\footnote{Actually, $G_{\mathcal{CS}}$ is bounded away from zero and from infinity; see [\ref{LapPo}] and [\ref{L-vF2}, Fig. 2.6 and \S2.3.1].}

Similarly, (\ref{tfCS3}) or (\ref{2.15}) becomes 
\begin{equation}\label{3}
\varepsilon^{-(1-D)}V_{\mathcal{CS}_3}(\varepsilon)=G_{\mathcal{CS}_3}(\log_3\varepsilon^{-1}),
\end{equation}
where
\begin{equation}\label{4}
G_{\mathcal{CS}_3}(x):= \frac{3^{-1}}{2\log3}\sum_{n\in \mathbb Z}
\frac{e^{2\pi inx}}{1-D-in\mathbf{p}}
\end{equation}
is a  nonconstant periodic function of period 1 on $\mathbb R$. 

In light of (\ref{1})--(\ref{2}) and (\ref{3})--(\ref{4}), it is clear that neither the limit (as $\varepsilon \rightarrow 0^+$) of 
$\varepsilon^{-(1-D)}V_{\mathcal{CS}}(\varepsilon)$
nor the limit of
$\varepsilon^{-(1-D)}V_{\mathcal{CS}_3}(\varepsilon)$
exists. 
Hence, neither the archimedean Cantor string $\mathcal{CS}$ nor the nonarchimedean Cantor string 
$\mathcal{CS}_3$ is Minkowski measurable (see  [\ref{Lap}, \ref{LapLu3}, \ref{LapPo}, \ref{L-vF2}] and \S\ref{Mdimension}).
However, by suitably averaging the left-hand side of (\ref{1}) over a large number of periods of $G_{\mathcal{CS}}$, one can show that the appropriately defined average Minkowski content of $\mathcal{CS}$ exists and is given by 
\[
\mathcal M_{av}(\mathcal {CS})=\frac{2^{-D}}{(1-D)\log2};
\]
cf. [\ref{L-vF2}, Rem. 8.35].\footnote{The minor discrepancy between the value of 
$\mathcal M_{av}(\mathcal {CS})$ given here and that of  [\ref{L-vF2}, Rem. 8.35] is due to the fact that $\mathcal {CS}$ is defined as in [\ref{L-vF2}, \S1.1.2]  and not as in [\ref{L-vF2}, \S2.3.1]; in particular, it has total length 1 rather than 3.}
Similarly, by suitably averaging the left-hand side of (\ref{3}) over infinitely many periods of $G_{\mathcal{CS}_3}$, one shows that the  average Minkowski content of $\mathcal{CS}_3$ exists and is given by 
\[
\mathcal M_{av}(\mathcal {CS}_3)=\frac{1}{6(\log3-\log2)};
\]
see Definition \ref{defcontent} and Example \ref{average content CS3} in \S\ref{amc} below.

\section{$p$-Adic Fractal Strings}\label{pfs}
 Let  $ \Omega$ be a bounded open subset of $\mathbb{Q}_p$.  Then it can be decomposed into a countable union of disjoint open balls \footnote{We shall often call a $p$-adic ball an \emph{interval}.  By `ball' here, we mean a metrically closed and hence, topologically open (and closed) ball.}
 with radius $p^{-n_j}$ centered at $a_j\in \mathbb Q_p$, 
 \[ a_j + p^{n_j} \mathbb{Z}_p=B(a_j, p^{-n_j})=\{x\in \mathbb{Q}_p ~|~ |x- a_j|_p \le p^{-n_j}\},\]
  where $n_j \in \mathbb{Z}$ and $j\in \mathbb N^{*}$. 
  There may be many different such decompositions since each ball can always be decomposed into smaller disjoint balls [\ref{Kob}]; see Equation (\ref{decomposition}).
 However, there is  a canonical decomposition of $\Omega$ into disjoint balls with respect to a suitable equivalence relation, as we now explain.
  
  \begin{definition}
 Let $U$ be an open subset of $\mathbb Q_p$. Given $x,y \in U,$ we write that $x\sim y$ if and only if there is a 
 ball $B\subseteq U$ such that $x, y \in B$.
 \end{definition}
 It is clear from the definition that the relation $\sim$ is reflexive and symmetric. To prove the transitivity, let $x \sim y$ and $y\sim z$. Then there are balls $B_1$ containing $x, y$ and $B_2$ containing $y, z$. Thus $y\in B_1 \cap B_2$; so it follows from the ultrametricity of $\mathbb Q_p$ that either 
  $B_1 \subseteq B_2$ or $B_2 \subseteq B_1.$   
  In any case, $x$ and $z$ are contained in the same ball; so $x\sim z$. Hence, the above relation
  $\sim$ is indeed an equivalence relation on the open set $U$. By a standard argument (and since $\mathbb Q$ is dense in $\mathbb Q_p$), one shows that there are at most countably many equivalence classes. 
  
  \begin{remark} \label{convex component}(Convex components)
  The equivalence classes of $\sim$ can be thought of as the `convex components' of $U$.  They are an appropriate substitute in the present nonarchimedean context for the notion of connected components, which is not useful in $\mathbb Q_p$ since $\mathbb Z_p$ (and hence, every interval) is totally disconnected.  Note that given any $x\in U,$ the equivalence class (i.e., the \emph{convex component}) of $x$ is the largest  ball containing $x$ (or equivalently, centered at $x$) and contained in $U$. 
    \end{remark}
  
   \begin{definition}\label{p-adic string}
 A \emph{$p$-adic} (or \emph{nonarchimedean}) fractal string
 $ \mathcal{L}_p$ is a bounded open subset $\Omega$ of $\mathbb{Q}_p$. 
 \end{definition}
 Thus it can be written, relative to the above equivalence relation, canonically as a disjoint union of  intervals or balls: 
 \[ \mathcal{L}_p=\bigcup_{j=1}^{\infty} (a_j + p^{n_j} \mathbb{Z}_p)=\bigcup_{j=1}^{\infty} B(a_j, p^{-n_j}). \]
 Here, $B(a_j, p^{-n_j})$ is the largest ball centered at $a_j$ and contained in $\Omega$.
We may assume that the lengths (i.e., Haar measure) of the intervals 
$a_j + p^{n_j} \mathbb{Z}_p$ are nonincreasing, by reindexing if necessary.  That is, 
\begin{equation}\label{sequence of lengths}
p^{-n_1}\geq p^{-n_2} \geq p^{-n_3} \geq \cdots >0.
\end{equation}

\begin{definition}\label{zetaLp}
The\emph{ geometric zeta function} of a $p$-adic fractal string $\mathcal{L}_p$ is defined as
\begin{equation} \label{zeta}
\zeta_{\mathcal{L}_p} (s) = \sum_{j=1}^{\infty} (\mu_H (a_j + p^{n_j} \mathbb{Z}_p))^s
= \sum_{j=1}^{\infty} p^{-n_js} 
\end{equation}
for $\Re(s)$ sufficiently large.
 \end{definition}

\begin{remark}
The geometric zeta function $\zeta_{\mathcal{L}_p}$ is well defined since the decomposition of 
${\mathcal{L}_p}$ into the disjoint intervals $a_j + p^{n_j}\mathbb Z_p$ is unique. Indeed, these intervals are the equivalence classes of which the open set $\Omega$ (defining $\mathcal L_p$) is composed. In other words, they are the $p$-adic  ``convex components'' (rather than the connected components) of $\Omega$. Note that in the real (or archimedean) case, there is no difference between the convex or connected components of $\Omega$, and hence the above construction would lead to the same sequence of lengths as in [\ref{L-vF2}, \S1.2].  
\end{remark}

\begin{figure}[h]\label{screen window}
\psfrag{dots}{$\dots$}
\psfrag{S}{$S$}
\psfrag{W}{$W$}
\psfrag{D}{$D$}
\psfrag{0}{$0$}
\psfrag{1}{$1$}
\raisebox{-1cm}
{\psfig{figure=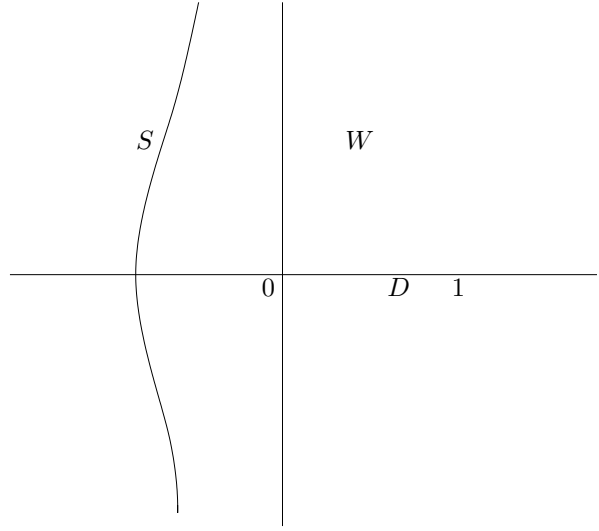, height=7cm}}
\caption{The screen $S$ and the window $W$.}
\end{figure}
  
The \emph{screen} $S$ is the graph \footnote{With the vertical and horizontal axes interchanged.} of a real-valued, bounded and Lipschitz continuous function $S(t)$:
\[
S=\{S(t) + it ~|~ t\in \mathbb R\}.
\]
The \emph{window} $W$ is the part of the complex plane to the right of the screen $S$: 
(see Figure 7): 
\[
W=\{s\in \mathbb C ~|~ \Re(s)\geq S(\Im (s))\}.
\]
Let 
\[
\inf S=\inf_{t\in \mathbb R} S(t) \quad \mbox{and} \quad \sup S=\sup_{t\in \mathbb R}S(t), 
\]
and assume that $\sup S \leq \sigma,$ where $\sigma=\sigma_{\mathcal L_p}$ is the abscissa of convergence of $\mathcal L_p$ (to be precisely defined in (\ref{sigma}) below).

  \begin{definition}\label{dvcd} 
  If $\zeta_{\mathcal{L}_p}$ has a meromorphic continuation to an open connected neighborhood of $W\subseteq \mathbb C$, then 
\begin{equation}\label{vcd}
 \mathcal D_{\mathcal L_p}(W)=\{\omega \in W ~|~ \omega \mbox{ is a pole of} ~ \zeta_{\mathcal{L}_p}\} 
 \end{equation}
is called the set of \emph{visible complex dimensions} of $\mathcal L_p$.  
If no ambiguity may arise or  if $W=\mathbb C$,  we simply write $\mathcal D_{\mathcal L_p}=\mathcal{D}_{\mathcal{L}_p}(W)$ and call it the set of \emph{complex dimensions} of $\mathcal L_p$. 

Moreover, the \emph{abscissa of convergence} of $\mathcal L_p$ (or rather, of the Dirichlet series initially defining $\zeta_{\mathcal L_p}$ in Equation (\ref{zeta})) is denoted by
$\sigma=\sigma_{\mathcal L_p}$.
Recall that it is defined by \footnote{See, e.g., [\ref{Ser}].} 
\begin{equation}\label{sigma}
\sigma_{\mathcal L_p}=\inf\left\{\alpha \in \mathbb R ~| ~ \sum_{j=1}^{\infty}p^{-n_j\alpha} <\infty\right\}.
\end{equation}
\end{definition}

\begin{remark}
In particular, if $\zeta_{\mathcal L_p}$ is entire (which occurs only in the trivial case when 
$\mathcal L_p$ is given by a finite union of intervals), then $\sigma_{\mathcal L_p}=-\infty.$
Otherwise, $\sigma_{\mathcal L_p}\geq 0$ (since $\mathcal L_p$ is composed of infinitely many intervals) and we will see in Theorem \ref{D=sigma} that $\sigma_{\mathcal L_p}<\infty$ since $\sigma_{\mathcal L_p}\leq D_M\leq 1,$ where $D_M=D_{M,\mathcal L_p}$ is the Minkowski dimension of $\mathcal L_p$.
Furthermore, it will follow from Theorem \ref{D=sigma} that for a nontrivial $p$-adic fractal  string, 
$\sigma_{\mathcal L_p}=D_M.$ This is the case, for example, for the 3-adic Cantor string introduced in 
\S\ref{CS3}, for which $\sigma=D_M=\log_32$.

Observe that since $ \mathcal D_{\mathcal L_p}(W)$ is defined as a subset of the poles of a meromorphic function, it is at most countable. 

Finally, we note that it is well known that $\zeta_{\mathcal L_p}$ is holomorphic for $\Re(s)>\sigma_{\mathcal L_p};$ see, e.g., [\ref{Ser}]. Hence, 
\[
\mathcal D_{\mathcal L_p} \subset \{s\in \mathbb C~|~ \Re(s)\leq \sigma_{\mathcal L_p}\}.
\]
\end{remark}

\begin{remark}[Archimedean fractal strings]\label{afs}
\emph{Archimedean} or \emph{real} fractal strings are defined as bounded open subsets of the real line $\mathbb R=\mathbb Q_{\infty}.$ They were initially defined in [\ref{LapPo}], following an early example in [\ref{Lap1}], and have been used extensively in a variety of settings; see, e.g., [\ref{Elmr}, 
\ref{HeLap}--\ref{HerLap},  \ref{Lap1}--\ref{Lap3}, \ref{Llr}, \ref{LapMa}--\ref{LapRo1}, \ref{Pe}] and the books [\ref{L-vF1}, \ref{L-vF2}, \ref{Lap2}].
Since an open set $\Omega\subset \mathbb R$ is canonically equal to the disjoint union of finitely or countably many open and bounded intervals (namely, its connected components), say $\Omega=\bigcup_{j=1}^{\infty} I_j,$ we may also describe a real fractal string by a sequence of lengths $\mathcal L=\{l_j\}_{j=1}^{\infty},$ where $l_j=\mu_L(I_j)$ is the length or 1-dimensional Lebesgue measure of the interval $I_j$, written in nonincreasing order:\footnote{A justification for this identification is provided by the formula for the volume $V_{\mathcal L}(\varepsilon)$ of $\varepsilon$-inner tubes of $\Omega$, as given by Equation (\ref{LPformula}) below.}
\[l_1\geq l_2\geq l_3 \geq \ldots.\]
Note that since $\mu_L(\Omega) < \infty, l_j \rightarrow 0$ as $j\rightarrow \infty$ (except in the trivial case when $\Omega$ consists of finitely many intervals).\footnote{Also observe that the 1-dimensional Lebesgue measure $\mu_L$ is nothing but the Haar measure on $\mathbb R=\mathbb Q_{\infty}$, normalized so that $\mu_L([0,1])=1.$} 

All the definitions given above for $p$-adic fractal strings have a natural counterpart for real fractal strings. For instance, the geometric zeta function of $\mathcal L$ is initially defined by 
\begin{equation}\label{zetaL}
\zeta_{\mathcal L}(s)=\sum_{j=1}^{\infty}(\mu_L(I_j))^s=\sum_{j=1}^{\infty} l_j^s,
\end{equation}
for $\Re(s) > \sigma_{\mathcal L},$ the abscissa of convergence of $\mathcal L$, and for a given screen $S$  and associated window $W$, the set 
$\mathcal D_{\mathcal L}=\mathcal D_{\mathcal L}(W)$ of visible complex dimensions of $\mathcal L$ is given exactly as in (\ref{vcd}) of Definition \ref{dvcd}, except with $\mathcal L_p$ and $\zeta_{\mathcal L_p}$
replaced with $\mathcal L$ and $\zeta_{\mathcal L}$, respectively. Similarly, $\sigma_{\mathcal L}$, the abscissa of convergence of $\mathcal L$ is given as in (\ref{sigma}), except with the lengths of $\mathcal L$ instead of those of  $\mathcal L_p.$
Moreover, it follows from [\ref{L-vF2}, Thm. 1.10] that for any nontrivial real fractal string $\mathcal L$, we have $\sigma_{\mathcal L}=D_M$, the Minkowski dimension of $\mathcal L$ (i.e., of its topological boundary $\partial \Omega$).

We refer the interested reader to the research monographs [\ref{L-vF1}, \ref{L-vF2}] for a full development of the theory of real fractal strings and their complex dimensions.
\end{remark}

\subsection{$p$-Adic Euler String}\label{Euler string}
The following $p$-adic Euler string  is a new example of $p$-adic fractal string, which is not self-similar (in the sense of \S \ref{sss}). It is a natural $p$-adic counterpart of  the \emph{elementary prime string},  which is the \emph{local} constituent of the \emph{completed  harmonic string}; cf. [\ref{L-vF2}, \S4.2.1].    

Let $X=p^{-1}\mathbb{Z}_p$. Then, by the `self-duplication' formula (\ref{decomposition}),
\[ 
 X= \bigcup_{\xi=0}^{p-1} (\xi p^{-1} + \mathbb{Z}_p).
  \]
We now keep the first subinterval $\mathbb{Z}_p$, and then decompose the next subinterval further. That is, we write
   \[ p^{-1} + \mathbb{Z}_p= \bigcup_{\xi=0}^{p-1} (p^{-1} + \xi  + p\mathbb{Z}_p).\]   
Again, iterating this process, we keep the first subinterval $p^{-1} + p\mathbb{Z}_p$ in the above decomposition and decompose the next subinterval,
   $p^{-1} + 1 + p\mathbb{Z}_p$. 
   Continuing in this fashion, we obtain an infinite sequence of disjoint subintervals 
   $\{ a_n + p^n \mathbb{Z}_p\}_{n=0}^{\infty},$ where
   $\{a_n\}_{n=0}^{\infty}$ satisfies the following initial condition and recurrence relation: 
   \[a_0 = 0 \quad  \mbox{and} \quad  a_n=a_{n-1} + p^{n-2} \quad \mbox{for all} ~n\geq 1.\] 
   We call the corresponding $p$-adic fractal string,
   \[
   \mathcal{E}_p=\bigcup_{n=0}^{\infty} (a_n + p^n \mathbb{Z}_p),
   \]
   the  \emph{$p$-adic Euler string.} (See Figure 8.)
 \begin{figure}[ht]
\psfrag{X}{$p^{-1}\mathbb Z_p$}    
\psfrag{hdot}{$\cdots$}
\psfrag{z1}{$\mathbb Z_p$}
\psfrag{z2}{$p^{-1}+\mathbb Z_p$}
\psfrag{z3}{$(p-1)p^{-1}+\mathbb Z_p$}
\psfrag{z4}{$p^{-1}+p\mathbb Z_p$}
\psfrag{z5}{$p^{-1}+1+p\mathbb Z_p$}
\psfrag{z6}{$p^{-1}+p-1+p\mathbb Z_p$}
\psfrag{vdot}{$\vdots$}
\raisebox{-1cm}
{\psfig{figure=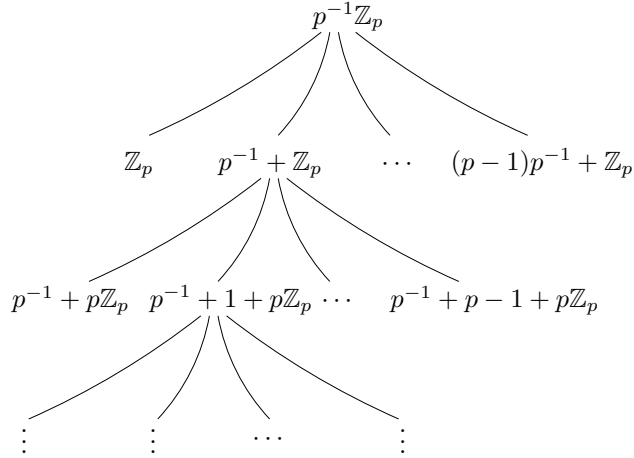, height=6 cm}}
\caption{Construction of the $p$-adic Euler string $\mathcal{E}_p$.}
\end{figure}

The geometric zeta function of the $p$-adic Euler string   $\mathcal{E}_p$  is
\begin{equation*}\label{euler factor}
\zeta_{\mathcal{E}_p} (s) = \sum_{n=0}^{\infty} (\mu_H( a_n + p^n \mathbb{Z}_p)) ^s 
= \sum_{n=0}^{\infty} p^{-ns}=
\frac{1}{1-p^{-s}},  \hspace{1cm} \mbox{for}\,\, \Re(s) > 0.
\end{equation*}
Therefore, $\zeta_{\mathcal{E}_p}$ has a meromorphic extension to all of $\mathbb C$ given by the last expression, which is the classic $p^{th}$-\emph{Euler factor}:
\begin{equation}\label{ptheuler}
\zeta_{\mathcal{E}_p} (s) =
\frac{1}{1-p^{-s}},  \hspace{1cm} \mbox{for  } \,s\in \mathbb C.
\end{equation}
 Hence, the set of complex dimensions of $\mathcal{E}_p$  is given by
\begin{equation}\label{cdes}
 \mathcal{D}_{\mathcal{E}_p}=\{ D + i\nu\textbf{p}~ |~ \nu \in \mathbb{Z}\}, 
 \end{equation}
  where $D=\sigma=0$ and $\textbf{p}=2\pi /{\log{p}}.$

\begin{remark}[Ad\`elic Euler string] 
Note   that  $\zeta_{\mathcal{E}_p}$ is  the $p^{th}$-Euler factor of the Riemann zeta function; i.e., 
\begin{equation*}\label{Riemann}
\prod_{p<\infty} \zeta_{\mathcal{E}_p} (s)
=\prod_{p<\infty} \frac{1}{1-p^{-s}}= \sum_{n=1}^{\infty} \frac{1}{n^s} =\zeta(s) \hspace{1cm}  \mbox{for} \,\,\Re(s) > 1.
\end{equation*}Recall that the meromorphic continuation $\xi$ of the Riemann zeta function $\zeta$ has the same (critical) zeros as $\zeta$ and satisfies the functional equation $  \xi (s)= \xi (1-s)$.

We hope to form a certain `ad\`elic product'   over all   $p$-adic Euler strings (including the prime at infinity) so that the geometric zeta function of the resulting ad\`elic Euler string $\mathcal E$ is the \emph{completed} Riemann zeta function.  Formally, the ad\`elic Euler string may be written as 
\[\mathcal{E}=\bigotimes_{p\leq \infty} \mathcal{E}_p\] and its geometric zeta function 
$\zeta_{\mathcal E}(s)$ would then coincide with the completed Riemann zeta function $\xi$
(see [\ref{Rie}] and, e.g., [\ref{Edw}]):
  \[ \zeta_{\mathcal E}(s)=\xi(s):=\pi^{-s/2}\Gamma(s/2)\prod_{p<\infty} \frac{1}{1-p^{-s}}.\] 
  \end{remark}
  
  \begin{remark}[Comparison with the archimedean theory]\label{harmonic string}
  From the geometric point of view, the nonarchimedean Euler string $\mathcal E_p$ is more natural than its archimedean counterpart, the $p^{th}$ elementary prime string $h_p$, described in 
  [\ref{L-vF2}, \S4.2.1]. 
Indeed, as we have just seen, $\mathcal E_p$ has a very simple geometric definition. Since, by construction, $\mathcal E_p$ and $h_p$ have the same sequence of lengths $\{p^{-n}\}_{n=0}^{\infty},$
they have the same geometric zeta function, namely, the $p^{th}$ Euler factor
\begin{equation}
\zeta_p(s):=\frac{1}{1-p^{-s}}
\end{equation}
of the Riemann zeta function $\zeta(s),$ and hence, the same set of complex dimensions
\begin{equation}
 \mathcal{D}_p=\left\{i\nu \frac{2\pi}{\log p}~ |~ \nu\in \mathbb{Z}\right\}.
 \end{equation}  
 An `ad\`elic version' of the `harmonic string' $h$, a generalized fractal string whose geometric zeta function is $\zeta_h(s)=\zeta(s)$,  or rather, of its completion $\tilde{h}$ 
 (so that $\zeta_{\tilde h}(s)=\xi(s)$), is provided in [\ref{L-vF2}, \S4.2.1].
 In particular, with each term being interpreted as a positive measure on $(0,\infty)$ and the symbol $\ast$ denoting multiplicative convolution on $(0,\infty)$, we have that 
 \begin{equation}
 h=*_{p<\infty}h_p \quad \mbox{and} \quad  \tilde h=\ast_{p\leq\infty}h_p. 
 \end{equation}
  
 Furthermore, a noncommutative geometric version of this construction is provided in [\ref{Lap2}] in terms of the `prime fractal membrane'; see especially, [\ref{Lap2}, Chaps. 3 and 4], along with [\ref{LapNe}].
 Heuristically, a `fractal membrane' (as introduced in [\ref{Lap2}]) is a kind of ad\`elic, noncommutative torus of infinite genus. It can also be thought of as a `quantized fractal string'; see
   [\ref{Lap2}, Chap. 3].   It is rigorously constructed in [\ref{LapNe}] using Dirac-type operators, Fock spaces, Toeplitz algebras, and associated spectral triples (in the sense of [\ref{Con}]); see also [\ref{Lap2}, \S4.2].
   We hope in the future to obtain a suitable nonarchimedean version of that construction. It is possible that in the process, we will establish contact with the physically motivated work in [\ref{Drag}] involving 
   $p$-adic quantum mechanics. 
 \end{remark}
  
\subsection{Volume of Thin Inner Tubes}\label{inner tube}
 In this section, based on a part of [\ref{LapLu3}], we provide a suitable analog in the $p$-adic case of the `boundary' of a fractal string and of the associated inner tubes 
 (or  ``inner $\varepsilon$-neighborhoods''). Moreover, we give the $p$-adic counterpart of the expression that yields the volume of the inner tubes (see Theorem \ref{thin}). This result serves as a starting point in [\ref{LapLu3}] for proving the corresponding explicit tube formula. 

\begin{definition}\label{volume definition}
Given a point $a\in \mathbb Q_p$ and a positive real number $r>0$,  
let $B=B(a, r) =\{x\in \mathbb{Q}_p ~|~ |x-a|_p \le r \}$ be a \emph{metrically closed} ball in $\mathbb{Q}_p,$ as above.\footnote{Recall that it follows from the ultrametricity of $|\cdot|_p$ that $B$ is topologically both closed and open (i.e., clopen) in $\mathbb Q_p$.}
 We  call 
$S=S(a, r)=\{x\in \mathbb{Q}_p ~|~ |x-a|_p = r \}$ the \emph{sphere} of $B$.\footnote{In our sense, $S$ also coincides with the `metric boundary' of $B$, as given in this definition.}

Let $\mathcal{L}_p= \bigcup_{j=1}^{\infty} B(a_j, r_j)$ be a $p$-adic fractal string. We then define the \emph{metric boundary} $\beta\mathcal{L}_p$ of $\mathcal{L}_p$ to be the disjoint union of the corresponding spheres, i.e.,  
 \[\beta\mathcal{L}_p = \bigcup_{j=1}^{\infty} S(a_j, r_j).\]
Given a real number $\varepsilon>0$,  define the  \emph{thick $p$-adic `inner 
$\varepsilon$-neighborhood'} (or `\emph{inner tube}') of $\mathcal{L}_p$ to be
\begin{equation}\label{thick inner tube}
\mathcal{N}_{\varepsilon}=\mathcal{N}_{\varepsilon}(\mathcal L_p):=\{ x\in \mathcal{L}_p ~|~  d_p(x, \beta\mathcal{L}_p) < \varepsilon\},
\end{equation}
where $d_p(x, E)=\inf \{ |x-y|_p ~|~  y\in E\}$
is the $p$-adic distance of $x \in \mathbb{Q}_p$ to a subset $E \subset \mathbb{Q}_p$.
Then the \emph{volume $\mathcal{V}_{\mathcal L_p}(\varepsilon)$ of the thick inner 
$\varepsilon$-neighborhood} of $\mathcal{L}_p$  is defined to be the Haar measure of $\mathcal{N}_{\varepsilon}$, i.e., 
$\mathcal{V}_{\mathcal L_p}(\varepsilon)=\mu_H(\mathcal{N}_{\varepsilon}).$
\end{definition}
 
\begin{lemma}\label{ball and sphere}
Let $B=B(a,r)$ and $S=S(a,r)$, as in Definition \ref{volume definition}. Then, for any positive number $\varepsilon <r$, we have 
\begin{equation}\label{sphere}
\mathcal{N}_{\varepsilon}(B):=\{x\in B ~|~ d_p(x, S)<\varepsilon\}=S. 
\end{equation}
Hence, if $r=p^{-m}$ for some $m\in \mathbb Z,$ then for all $\varepsilon <r,$
\begin{equation}\label{spherevolume}
\mu_H(\{x\in B ~|~ d_p(x, S)<\varepsilon\})=\mu_H(S)=
(1-p^{-1})p^{-m}.
\end{equation}
\end{lemma}

\begin{proof}
(i) Clearly $S \subseteq\{x\in B ~|~ d_p(x, S)<\varepsilon\}$   since for any $x\in S$, 
$d_p(x,S)=0$. Next, fix $\varepsilon$ with $0<\varepsilon<r$ and let $x\in B$ be such that $d_p(x, S)< \varepsilon.$ Then there must exist $y\in S$ such that $|x-y|_p<\varepsilon.$
But, since $|y-a|_p=r,$ we deduce from the fact that every ``triangle'' in $\mathbb Q_p$ is isosceles [\ref{Kob}, p. 6] that 
$|x-a|_p=|y-a|_p$ and thus $x\in S$. This completes the proof of (\ref{sphere}).  \\

(ii) We next establish formula (\ref{spherevolume}). In light of Equation (\ref{sphere}), it suffices to show that 
\begin{equation}\label{Svolume}
\mu_H(S)=(1-p^{-1})p^{-m}.
\end{equation}
Let $S^1=S(0,1)=\{x\in \mathbb Q_p~|~ |x|_p=1\}$ denote the unit sphere in $\mathbb Q_p$. Since 
$S=S(a, p^{-m})=a+p^mS^1,$ we have that $\mu_H(S)=\mu_H(S^1)p^{-m}.$
Next we note that 
\[B(0,1)= \bigcup_{m\geq 0}S(0,p^{-m})\] 
is a disjoint union. Hence,  by taking the Haar measure of $B(0,1),$ we deduce that 
\begin{equation}
  1 = \left(\sum _{m=0} ^ {\infty} p^{-m}\right) \mu_{H}(S^1) =  \frac{1}{1-p^{-1}} \mu_{H}(S^1),
  \end{equation}
from which  (\ref{Svolume}) and hence, in light of part (i), (\ref{spherevolume}) follows. 
\end{proof}

\begin{theorem}[Volume of thick inner tubes]\label{thick}
Let $ \mathcal{L}_p=\bigcup_{j=1}^{\infty}  B(a_j, p^{-n_j})$ be a $p$-adic fractal string. Then, for any $\varepsilon >0,$ we have
\begin{eqnarray}\label{volumeequation}
\mathcal{V}_{\mathcal L_p}(\varepsilon) &=&(1-p^{-1})\sum_{j=1}^k p^{-n_j} +\sum_{j>k} p^{-n_j}\label{equation 1}\\
&=& \zeta_{\mathcal L_p}(1)-\frac{1}{p} \sum_{j=1}^k   p^{-n_j},\label{equation 2}
    \end{eqnarray} 
  where $k=k(\varepsilon)$ is the largest   integer such that
 $p^{-n_k}\ge \varepsilon$.
 \end{theorem}

 \begin{proof}[Sketch of the proof]
In light of the definition of $\mathcal N_{\varepsilon}=\mathcal N_{\varepsilon}(\mathcal L_p)$ given in Equation (\ref{thick inner tube}) and the definition of $k$ given in the theorem, we have that 
\[
\mathcal N_{\varepsilon}=\bigcup_{j=1}^kS_j\cup \bigcup_{j>k}B_j,
\]
where $B_j:=B(a_j,p^{-n_j})$ and $S_j:=S(a_j,p^{-n_j})$ for each $j\geq 1$.

We then apply Lemma \ref{ball and sphere} to deduce the expression of 
$\mathcal V_{\mathcal L_p}(\varepsilon)=\mu_H(\mathcal N_{\varepsilon})$
stated in Equations (\ref{equation 1}) and (\ref{equation 2}).
\end{proof}

Note that 
$\zeta_{\mathcal L_p}(1)=\sum_{j=1}^{\infty}p^{-n_j}$ 
is the volume of $\mathcal L_p$ (or rather, of the bounded open subset $\Omega$ of $\mathbb Q_p$ representing $\mathcal L_p$):
\[
\zeta_{\mathcal L_p}(1)=\mu_H(\mathcal L_p)<\infty.
\]
It is clearly independent of the choice of $\Omega$ representing $\mathcal L_p$, and so is $\mathcal V_{\mathcal L_p}(\varepsilon)$ in light of either (\ref{equation 1}) or (\ref{equation 2}).

 \begin{corollary}\label{limit}
 The following limit exists in $(0,\infty):$
 \begin{equation}
 \lim_{\varepsilon \to 0^+} \mathcal{V}_{\mathcal L_p}(\varepsilon)
 =\mu_H(\beta\mathcal L_p)
 =(1-p^{-1})\zeta_{\mathcal L_p}(1).
   \end{equation} 
 \end{corollary} 
 
 This follows by letting $\varepsilon \rightarrow 0^+$ in either  (\ref{equation 1}) and (\ref{equation 2}) and noting that $k=k(\varepsilon)\rightarrow \infty.$
 
Corollary \ref{limit}, combined with the fact that $\beta\mathcal L_p \subset \mathcal N_{\varepsilon}(\mathcal L_p)$ for any $\varepsilon >0,$ naturally  leads us to introduce the following definition. 

\begin{definition}\label{thin inner tube}
Given $\varepsilon >0,$ the \emph{thin $p$-adic `inner $\varepsilon$-neighborhood'} (or \emph{ `inner tube'}) of 
$\mathcal L_p$ is given by 
\begin{equation}\label{equation N}
N_{\varepsilon}=N_{\varepsilon}(\mathcal L_p):=\mathcal N_{\varepsilon}(\mathcal L_p)\backslash \beta\mathcal L_p.
\end{equation}

Then, in light of Corollary \ref{limit}, the \emph{volume $V_{\mathcal L_p}(\varepsilon)$ of the thin inner  
$\varepsilon$-neighborhood} of $\mathcal L_p$ is defined to be the Haar measure of $N_{\varepsilon}$ and is given by
\begin{equation}\label{volume of the thin inner tube}
V_{\mathcal L_p}(\varepsilon):=\mu_H(N_{\varepsilon})
=\mathcal V_{\mathcal L_p}(\varepsilon)-\mu_H(\beta \mathcal L_p).
\end{equation}
Note that, by construction, we now have $\lim_{\varepsilon \to 0^+}V_{\mathcal L_p}(\varepsilon)=0.$
\end{definition}

We next state the counterpart (for thin inner tubes) of Theorem \ref{thick}, which is the key result that will enable us to obtain an appropriate $p$-adic analog of the fractal tube formula as well as of the notion of Minkowski dimension and content (see \S\ref{Mdimension}).

\begin{theorem}[Volume of thin inner tubes]\label{thin}
Let $ \mathcal{L}_p=\bigcup_{j=1}^{\infty}  B(a_j, p^{-n_j})$ be a $p$-adic fractal string. Then, for any $\varepsilon >0,$ we have
\begin{eqnarray}\label{volumeequation}
V_{\mathcal L_p}(\varepsilon) &=&p^{-1}\sum_{j>k} p^{-n_j} =p^{-1}\sum_{j:p^{-n_j}<\varepsilon} p^{-n_j}\\
&=& p^{-1}\left(\zeta_{\mathcal L_p}(1)-\sum_{j=1}^k   p^{-n_j}\right),\label{equation2}
    \end{eqnarray} 
  where $k=k(\varepsilon)$ is the largest   integer such that
 $p^{-n_k}\ge \varepsilon$, as before.
\end{theorem}

\begin{remark}
Observe that because the center $a$ of a $p$-adic ball $B=B(a,p^{-n})$ can be chosen arbitrarily without changing its radius $p^{-n},$ the metric boundary of a ball, $\beta B=S=S(a,p^{-n}),$ may depend on the choice of $a$. 
Note, however, that in view of 
Equation (\ref{spherevolume}) in Lemma \ref{ball and sphere}, its volume $\mu_H(S)$ depends only on the radius of $B$.
Similarly, even though the decomposition of a $p$-adic fractal string $\Omega$ (i.e., $\mathcal L_p$) into maximal balls $B_j=B_j(a_j, p^{-n_j})$ is canonical, `the' metric boundary of $\mathcal L_p,$
 $\beta\mathcal{L}_p = \bigcup_{j=1}^{\infty} S(a_j, r_j),$ may in general depend on the choice of the centers $a_j.$ However, according to Corollary \ref{limit}, $\mu_H(\beta\mathcal L_p)$ is independent of this choice and hence, neither 
 $\mathcal V_{\mathcal L_p}(\varepsilon)=\mu_H(\mathcal N_{\varepsilon}(\mathcal L_p))$ nor 
 $V_{\mathcal L_p}(\varepsilon)=\mu_H(N_{\varepsilon}(\mathcal L_p))$ 
 depends on the choice of the centers. 
 Indeed, in light of Theorem \ref{thick} and Theorem \ref{thin}, 
 $\mathcal V_{\mathcal L_p}(\varepsilon)$ and
 $V_{\mathcal L_p}(\varepsilon)$
 depend only on the choice of the $p$-adic lengths $p^{-n_j},$ and hence solely on the $p$-adic fractal string $\mathcal L_p$, viewed as a nonincreasing sequence of positive numbers, and not on the representation $\Omega,$ let alone on the choice of the centers of the balls of which $\Omega$ is composed.
 Although it is not entirely analogous to it, this situation is somewhat  reminiscent of the fact that the volume $V_{\mathcal L}(\varepsilon)$ of the inner $\varepsilon$-neighborhoods of an archimedean fractal string depends only on its lengths $\{l_j\}_{j=1}^{\infty}$ and not on the representative 
 $\Omega$ of $\mathcal L$ as a bounded open set; see Equation (\ref{LPformula}) and the discussion surrounding it in Remark \ref{real vs. p-adic volume}. 
 \end{remark}

 \begin{remark}[Comparison between the archimedean and the nonarchimedean cases]\label{real vs. p-adic volume}
        Recall that $\mathcal V_{\mathcal L_p}(\varepsilon)$ does not tend to zero as 
          $\varepsilon \to 0^+,$ but that instead it tends to the positive number $(1-p^{-1})\zeta_{\mathcal L_p}(1),$ whereas $V_{\mathcal L_p}(\varepsilon)$ does tend to zero. This is the reason why the Minkowski dimension must be defined in terms of $V_{\mathcal L_p}(\varepsilon)$ (as will be done in \S \ref{Mdimension}) rather than in terms 
          of $\mathcal V_{\mathcal L_p}(\varepsilon).$
                    Indeed, if $\mathcal V_{\mathcal L_p}(\varepsilon)$ were used instead, then every $p$-adic fractal string would have Minkowski dimension 1. This would be the case even for a trivial $p$-adic fractal string composed of a single interval, for example. This is also why, in the $p$-adic case, we will focus only on the tube formula for $V_{\mathcal L_p}(\varepsilon)$
                    rather than  for   $\mathcal V_{\mathcal L_p}(\varepsilon)$, although the latter could be obtained by means of the same techniques.   
                                
        Note the difference between the expressions for $V_{\mathcal L}(\varepsilon)$ in the case of an archimedean fractal string $\mathcal L$ and for its nonarchimedean thin (resp., thick) counterpart 
    $V_{\mathcal L_p}(\varepsilon)$ (resp., $\mathcal V_{\mathcal L_p}(\varepsilon)$)
    in the case of a $p$-adic fractal string $\mathcal L_p$. 
    Compare Equation (8.1) of [\ref{L-vF2}] (which was first obtained in [\ref{LapPo}]),
    \begin{equation}\label{LPformula}
    V_{\mathcal L}(\varepsilon)=\sum_{j: l_j\geq 2\varepsilon}2\varepsilon +\sum_{j: l_j< 2\varepsilon}l_j,
    \end{equation}
    with Equations (\ref{volumeequation})--(\ref{equation2})  in Theorem \ref{thin}. (Here, we are using the notation of Remark \ref{afs}, to which the reader is referred to for a brief introduction to real fractal strings.) 
          It follows, in particular, that   $V_{\mathcal L}(\varepsilon)$ is a continuous function of $\varepsilon$ on $(0,\infty)$, whereas  $\mathcal V_{\mathcal L_p}(\varepsilon)$ (and hence also $V_{\mathcal L_p}(\varepsilon)$) is discontinuous (because it is a step function with jump discontinuities at each point $p^{-n_j}, \mbox{for}~ j=1, 2, \ldots $).    
           The above discrepancies between the archimedean and the nonarchimedean cases help explain why the tube formula for real and $p$-adic fractal strings have a similar form, but with different expressions for the corresponding `tubular zeta function' (in the sense of [\ref{LapPe1}--\ref{LPW}]). We note that a minor aspect of these discrepancies is that $2\varepsilon$ is now replaced by $\varepsilon.$ Interestingly, this is due to the fact that the unit interval $[0,1]$ has inradius 1/2 in $\mathbb R=\mathbb {Q}_{\infty}$ whereas 
          $\mathbb Z_p$ has inradius 1 in $\mathbb Q_p.$\footnote{Recall that the \emph{inradius} of a subset $E$ of a metric space is the supremum of the radii of the balls entirely contained in $E$.}
          
          Finally, we note that for an archimedean fractal string $\mathcal L$, there is no reason to distinguish between the `thin volume' $V_{\mathcal L}$ and the `thick volume' 
          $\mathcal V_{\mathcal L}$, as we now explain. Indeed, the archimedean analogue $\beta \mathcal L$ of the metric boundary is a countable set, and hence has measure zero, no matter which realization $\Omega$ one chooses for $\mathcal L$.
          More specifically, in the notation of Remark \ref{afs}, $\beta \mathcal L$ consists of all the endpoints of the open intervals $I_j$ (the connected components of $\Omega$, or equivalently, its convex components). Hence, $\mu_L(\beta \mathcal L)=0$ and so 
          \[V_{\mathcal L}(\varepsilon):=\mathcal V_{\mathcal L}(\varepsilon)-\mu_L(\beta \mathcal L)=
          \mathcal V_{\mathcal L}(\varepsilon),\]
          as claimed. 
         
          For example, if $\mathcal L$ is the ternary Cantor string $\mathcal {CS}$, then 
          $\beta \mathcal L$ is the countable set consisting of all the endpoints of the `deleted intervals'
  in the construction of the real Cantor set $\mathcal {C}$ (see \S\ref{C}); in other words,   $\beta\mathcal L$   is the set $\mathcal T$ of ternary points (which has measure zero because it is countable).
  Hence, the metric boundary $\beta \mathcal L$ of $\mathcal{CS}$  is dense in 
  $\partial \mathcal L$, the topological boundary of $\mathcal {CS}$, and which in the present case, coincides with the ternary Cantor set $\mathcal C$. Also note that the fact that 
  $\mathcal C=\partial \mathcal L$ (and not $\mathcal T=\beta\mathcal L$) has measure zero is purely coincidental and completely irrelevant here. 
          Indeed, the same type of argument would apply if $\mathcal L$ were any archimedean fractal string, even if $\mu_L(\partial \mathcal L)>0$ as is the case for example, if $\partial\mathcal L$ is a `fat Cantor set' (i.e., a Cantor set of positive measure) or, more generally, if $\partial\mathcal L$ is a `fat fractal' (in the sense of [\ref{GMOY}, \ref{Ott}]).
          The underlying reason is that in the archimedean case, the topological boundary $\partial \mathcal L=\partial \Omega$ is disjoint from $\Omega$  (since $\Omega$ is open), and hence, does not play any role in the computation of $V_{\mathcal L}(\varepsilon)$ or of $\mathcal V_{\mathcal L}(\varepsilon)$. 
          By contrast, it is not  true that the metric boundary  $\beta \mathcal L$ and $\Omega$ are disjoint (since, in fact,  $\beta \mathcal L \subset \Omega$), but what is remarkable is that the Minkowski dimension of $\beta \mathcal L$ coincides with that of its closure, and hence (in most cases of interest),
           with $D_{M,\mathcal L}$; see [\ref{Lap1}] and the relevant references therein.
                     \end{remark}   
            
 As a first application of Theorem \ref{thin}, we can obtain, via a direct computation, a tube formula for the $p$-adic Euler string $\mathcal E_p$; that is, an explicit formula for the volume of the thin inner 
 $\varepsilon$-neighborhood, $V_{\mathcal E_p}(\varepsilon)$, as given in Definition \ref{thin inner tube}. 
     Later on, we will similarly obtain a tube formula for the nonarchimedean counterpart of the Cantor string, as discussed in \S\ref{CS3}. 
     We will also show how to recover these results from the general theory developed in the next section. 
     
\begin{example}(Explicit  and exact  tube formula for the $p$-adic Euler string $\mathcal{E}_p$).\label{euler volume}
Let $\mathcal E_p$ be the $p$-adic Euler string defined in \S\ref{Euler string}.  Given  
  $\varepsilon >0$, let $k$ be the largest integer such that $\mu_H(a_k+p^k\mathbb Z_p)=p^{-k}\geq \varepsilon$; then $k=[\log_p \varepsilon^{-1}]. $\footnote{Here, for $x\in \mathbb R$, we write $x=[x]+\{x\},$ where $[x]$ is the integer part  and $\{x\}$ is the fractional part of 
$x$; i.e., $x\in \mathbb Z$ and $0\leq x<1$.}
Thus, by Equation (\ref{volumeequation}) of Theorem \ref{thin}, we have successively: 

\begin{eqnarray}
V_{\mathcal E_p}(\varepsilon)&=&p^{-1}\sum_{n=k+1}^{\infty} p^{-n}\nonumber\\
&=&\frac{p^{-1}}{p-1}p^{-k} \nonumber\\
&=& \frac{p^{-1}}{p-1}p^{-\log_p \varepsilon^{-1} }
  \left(\frac{1}{p}\right)^{-\{\log_p\varepsilon^{-1}\}}\label{V(epsilon)}\nonumber \\
  &=& \frac{p^{-1}}{p-1}  \frac{p-1}{\log p}\sum_{n\in \mathbb Z}\frac{\varepsilon^{1 -in\textbf p}}{1 -in\textbf p}\nonumber  \\
  &=&  \frac{1}{p\log p}\sum_{\omega \in \mathcal D_{\mathcal E_p} } 
     \frac{\varepsilon^{1-\omega}}{ 1-\omega }.\label{VolumeEuler}
    \end{eqnarray}
    We now explain some of the steps above. In the third equality, we have written that 
    $k=\log_p \varepsilon^{-1} -\{\log_p \varepsilon^{-1} \}.$  
     Furthermore, in  the next to last equality, we have appealed to the Fourier series expansion for $b^{-\{x\}}$  given by 
    \begin{equation}\label{fourier}
    b^{-\{x\}}=\frac{b-1}{b}\sum_{n\in \mathbb Z}\frac{e^ {2\pi i n x}}{\log b+2\pi in},
    \end{equation}
    for $b=p^{-1}$ and $x=\log_p\varepsilon ^{-1}$. (See [\ref{L-vF2}, Eq. (1.13)].)
    Finally, in the last equality, we have used Equation (\ref{cdes}) for the set of complex dimensions $\mathcal D_{\mathcal E_p}$ of $\mathcal E_p$. 
        \end{example}
 
 \subsection{Minkowski Dimension}\label{Mdimension}  
 In the sequel, the (inner) Minkowski dimension and the (inner) Minkowski content of a $p$-adic fractal string $\mathcal L_p$ (or, equivalently, of its metric boundary $\beta \mathcal L_p$, see Definition
  \ref{volume definition}) is defined exactly as the corresponding notion for a real fractal string 
  (see [\ref{L-vF2}, Defn. 1.2]), 
 except for the fact that we use the definition of $V(\varepsilon)=V_{\mathcal L_p}(\varepsilon)$ provided in Equation (\ref{volume of the thin inner tube}) of \S\ref{inner tube}.\footnote{Recall that as was explained towards the end of Remark \ref{real vs. p-adic volume}, in the archimedean case, $V(\varepsilon)=V_{\mathcal L}(\varepsilon)$ is the same, whether it is defined by the analog for $\mathcal L$ of (\ref{V_CS}) or by the counterpart of (\ref{volume of the thin inner tube}) in Definition  \ref{thin inner tube}.}
  (For reasons that will be clear to the reader later on in this section, we denote by $D_M=D_{M,\mathcal L_p}$ instead of by $D=D_{\mathcal L_p}$ the Minkowski dimension of $\mathcal L_p.$) More specifically,
 the \emph{Minkowski dimension} of $\mathcal L_p$ is given by 
\begin{equation}\label{dimension}
D_M= D_{M,\mathcal L_p}:= \inf  \left\{\alpha \geq  0~|~ V_{\mathcal L_p}(\varepsilon)=O(\varepsilon ^{1-\alpha}) ~\mbox{as}~
\varepsilon \rightarrow 0^{+}  \right\}.  \end{equation}
Furthermore, $\mathcal L_p$ is said to be \emph{Minkowski measurable}, with \emph{Minkowski content}
$\mathcal M$, if the limit 
\begin{equation}
\mathcal M=\lim_{\varepsilon\rightarrow 0^{+}} V_{\mathcal L_p}(\varepsilon)\varepsilon ^{-(1-D_M)}
\end{equation}
exists in  $(0, \infty).$

\begin{remark}
Note that since 
$V_{\mathcal L_p}(\varepsilon)=\mathcal V_{\mathcal L_p}(\varepsilon)-\mu_H(\beta \mathcal L_p)$,
the above definition of the Minkowski dimension is somewhat analogous to that of ``exterior dimension'', which is sometimes used in the archimedean case to measure the roughness of a `fat fractal' (i.e., a fractal with positive Lebesgue measure). The notion of exterior dimension has been useful in the study of aspects of chaotic nonlinear dynamics; see, e.g., [\ref{GMOY}] and the survey article [\ref{Ott}].

\end{remark}

The following theorem (from [\ref{LapLu3}]) is the exact $p$-adic  analog of [\ref{L-vF2},  Thm. 1.10]
(first observed in [\ref{Lap}], using a result of Besicovitch and Taylor [\ref{BT}]).\footnote{Note that like in [\ref{L-vF2}, Thm. 1.10], we need to assume that $\mathcal L_p$ has infinitely many lengths since in the latter case, we have $\sigma_{\mathcal L_p}=-\infty.$}

\begin{theorem}\label{D=sigma}
Let $\mathcal{L}_p$ be a nontrivial $p$-adic fractal string.\footnote{i.e., $\mathcal L_p$ is not given by a finite union of intervals.} 
  Then the abscissa of convergence $\sigma_{\mathcal L_p}$ of the geometric zeta function $\zeta_{\mathcal{L}_p}$ coincides with the Minkowski dimension $ D_M.$ That is, $ \sigma_{\mathcal L_p}=D_M .$
\end{theorem}
 
 \begin{remark}\label{0D1}
 For any $p$-adic fractal  string, we have $0\leq D_M\leq 1.$ Indeed, by definition, $D_M\geq 0$; furthermore, in light of (\ref{dimension}), $D_M\leq 1$ since a $p$-adic fractal string is a bounded open set, and hence, has finite volume.
 \end{remark}
 
 The next corollary follows by combining Theorem \ref{D=sigma} and Remark \ref{0D1}.
 \begin{corollary}\label{corollary sigma}
 Let $\mathcal L_p$ be a nontrivial $p$-adic fractal string. Then $0\leq \sigma\leq 1$.
 \end{corollary}
 
\begin{remark}
The proof of $\sigma_{\mathcal L_p}\leq D_M$ in Theorem \ref{D=sigma} is obtained by adapting the first part of the proof of  [\ref{L-vF2}, Thm. 1.10]. In light of the new form of $V_{\mathcal L_p}(\varepsilon)$ given by Theorem \ref{thin}, however, it does not seem possible to prove the reverse inequality 
$\sigma_{\mathcal L_p}\geq D_M$ by simply adapting the second part of the proof of [\ref{L-vF2}, Thm. 1.10]. 
Nevertheless, this latter inequality follows from the following new integral representation of  the geometric zeta function
$\zeta_{\mathcal L_p}$ in Lemma \ref{L:zeta=N,V}, a specification to   $p$-adic fractal strings of a general lemma in [\ref{LapLu3}].
\end{remark}

\begin{lemma}\label{L:zeta=N,V}
Let $\mathcal L_p$ be a $p$-adic fractal string, then 
\begin{equation}\label{E:zeta=V}
\zeta_{\mathcal L_p}(s)=\zeta_{\mathcal L_p}(1)l_1^{s-1}+p(1-s)\int_0^{l_1}V_{\mathcal L_p}(\varepsilon)\varepsilon^{s-2}\,d\varepsilon,
\end{equation}
where $l_1=p^{-n_1}$ as in (\ref{sequence of lengths}). Furthermore, the integral converges exactly when $\zeta_{\mathcal L_p}(s)=\sum_{j=1}^{\infty}p^{-n_js}$ converges. 
\end{lemma}

\subsection{Languid and Strongly Languid $p$-Adic Fractal Strings}
In \S\ref{explicit tf}, we will obtain  explicit tube formulas for $p$-adic fractal strings, with and without error term. (See Theorem \ref{dtf} and Corollary \ref{ftf}.) We will then apply the  tube formula without error term (the strongly languid case of Theorem \ref{dtf}) to $p$-adic self-similar strings in \S\ref{exact tf}, and will apply its corollary to the $p$-adic Euler string discussed in \S\ref{Euler string} and revisited in Example \ref{eulerrevisit} (at the end of \S\ref{explicit tf}). 

In order to state the explicit formulas with (or without) error term, we need to assume the following technical hypotheses (see [\ref{L-vF2}, Defns. 5.2 and 5.3] and recall the definition of the screen $S$ given in \S\ref{pfs}, just before Definition \ref{dvcd}).

\begin{definition}\label{languid}
A $p$-adic  fractal string $\mathcal{L}_p$ is said to be \emph{languid} if its geometric zeta function 
$\zeta_{\mathcal{L}_p}$ satisfies the following growth conditions: There exist real constants $\kappa$ and $C>0$ and a two-sided sequence $\{T_n\}_{n\in \mathbb Z}$ of real numbers such that $T_{-n}<0<T_n$ for $n\geq1$, and 
\[
\lim_{n \rightarrow \infty} T_n=\infty, \quad  \lim_{n \rightarrow \infty} T_{-n}=-\infty, \quad
 \lim_{n \rightarrow \infty} \frac{T_n}{|T_{-n}|}=1,
 \]
 such that
 \begin{itemize}
 \item $\mathbf{L1}$ For all $n\in \mathbb Z$ and all $u \geq S(T_n),$
\[| \zeta_{\mathcal{L}_p}(u +i T_n)|\leq C (|T_n|+1)^{\kappa},\]
  \item $\mathbf{L2}$
 For all $t\in \mathbb R, |t|\geq 1,$
\[ | \zeta_{\mathcal{L}_p}(S(t)+it)|\leq C |t|^{\kappa}.\]
 \end{itemize}

We say that $\mathcal{L}_p$ is \emph{strongly languid} if its geometric zeta function $\zeta_{\mathcal{L}_p}$ satisfies the following conditions, in addition to $\mathbf{L1}$ with $S(t)\equiv -\infty:$
 There exists a sequence of screens $S_m: t\mapsto S_m(t)$ for $m\geq1, t\in \mathbb R$, with 
 $\sup S_m \rightarrow -\infty$ as $m\rightarrow \infty$ and with a uniform Lipschitz bound $\sup_{m\geq1} ||S_m||_{Lip}< \infty$, such that
 \begin{itemize}
 \item $\mathbf{L2'}$ There exist constants $A,C>0$ such that for all $t\in \mathbb R$ and $m\geq1$, 
 \[ |\zeta_{\mathcal{L}_p}(S_m(t)+it)|\leq C A^{|S_m(t)|}(|t| + 1)^{\kappa}.\]
 \end{itemize}
 \end{definition}
 
\begin{remark} (a) Intuitively, hypothesis $\mathbf{L1}$ is a polynomial growth condition along horizontal lines (necessarily avoiding the poles of $\zeta_{\mathcal{L}_p}$), while hypothesis $\mathbf{L2}$ is a polynomial growth condition along the vertical direction of the screen.

(b) Clearly, condition $\mathbf{L2'}$ is stronger than $\mathbf{L2}$. Therefore, if $\mathcal{L}_p$ is strongly languid then it is also languid (for each screen $S_m$ separately). 

(c) Moreover, if $\mathcal{L}_p$ is languid for some $\kappa$, then it is also languid for every larger value of $\kappa$.  The same is also true for strongly languid strings. 

(d) Finally, hypotheses $\mathbf{L1}$ and $\mathbf{L2}$ require that  $\zeta_{\mathcal{L}_p}$ has an analytic (i.e., meromorphic) continuation to an open, connected neighborhood of $\Re(s)\geq \sigma_{\mathcal L_p}$, while $\mathbf{L2'}$ requires that $\zeta_{\mathcal L_p}$ has a meromorphic continuation to all of $\mathbb C$. 
\end{remark}

 \subsection{Explicit Tube Formulas for $p$-Adic Fractal Strings} \label{explicit tf} 
 
 The following result is the counterpart in this context of Theorem 8.1 of [\ref{L-vF2}], the distributional tube formula for real fractal strings.  It is established in [\ref{LapLu3}] by using, in particular, the extended distributional explicit formula of [\ref{L-vF2}, Thms. 5.26 and 5.27], along with the expression for the volume of thin inner $\varepsilon$-tubes obtained in Theorem \ref{thin}.

\begin{theorem}[$p$-Adic explicit  tube formula]\label{dtf}
(i) Let $\mathcal L_p$ be a languid $p$-adic fractal string  for some real exponent $\kappa$ and a screen $S$ that lies strictly to the left of the vertical line $\Re(s)=1$.  
Further assume that $\sigma_{\mathcal L_p}<1.$\footnote{Recall from Corollary \ref{corollary sigma} that we always have $\sigma_{\mathcal L_p}\leq 1$. Moreover, we will see in Remark \ref{D<1} that if $\mathcal L_p$ is self-similar, then $\sigma_{\mathcal L_p}< 1$.}
Then the volume of the thin inner  $\varepsilon$-neighborhood of $\mathcal L_p$ is given by the following distributional explicit formula, on test functions in $\mathbf{D}(0, \infty):$\footnote{Here, $\mathbf{D}(0, \infty)$ is the space of $C^{\infty}$ functions with compact support in $(0,\infty)$.}
\begin{equation}\label{detf}
V_{\mathcal L_p}(\varepsilon)= \sum_{ \omega \in \mathcal{D}_{\mathcal{L}_p} (W)} res 
\left (
\frac{p^{-1} \zeta_{\mathcal{L}_p} (s) \varepsilon^{1-s}} {1-s}; \omega
\right) 
+ \mathcal{R}_p (\varepsilon),
\end{equation}
where $\mathcal D_{\mathcal L_p}(W)$ is the set of visible complex dimensions of $\mathcal L_p$ (as given in Definition \ref{dvcd}).
Here, the distributional error term is given by 
\begin{equation}\label{error term}
\mathcal R_p(\varepsilon)=\frac{1}{2\pi i}
\int_{S}\frac{p^{-1} \zeta_{\mathcal{L}_p} (s) \varepsilon^{1-s}} {1-s}ds
\end{equation}
and is estimated distributionally \footnote{As in [\ref{L-vF2}, Defn. 5.29].} by
\begin{equation}\label{estimate}
\mathcal{R}_p(\varepsilon)=O(\varepsilon^{1-\sup S}), \hspace{1cm} \mbox{as} ~\varepsilon \rightarrow 0^+.
\end{equation}

(ii) Moreover, if $\mathcal L_p$ is strongly languid (as in the second part of Definition \ref{languid}), then we can take $W=\mathbb C$ and
$\mathcal{R}_p(\varepsilon)\equiv 0$, provided we apply this formula to test functions supported on   compact subsets  of $[0, A)$. The resulting explicit formula without error term is often called an \emph{exact tube formula} in this case. 
\end{theorem}

\begin{remark}\label{tzf}
We may rewrite the (typically infinite) sum in (\ref{detf}) as follows:
\begin{equation}\label{sum residue}
\sum_{ \omega \in \mathcal{D}_{\mathcal{L}_p} (W)} res(
\zeta_{\mathcal L_p}(\varepsilon; s); s=\omega), 
\end{equation}
where (by analogy with the definitions and results in [\ref{LapPe1}--\ref{LPW}]),
\begin{equation}\label{tubular zeta}
\zeta_{\mathcal L_p}(\varepsilon; s):=\frac{p^{-1}\zeta_{\mathcal L_p}(s)\varepsilon^{1-s}}{1-s}
\end{equation}
is called the \emph{nonarchimedean tubular zeta function} of the $p$-adic fractal string $\mathcal L_p$.

By contrast, the archimedean tubular zeta function (in the present one-dimensional situation) of a real fractal string $\mathcal L$ is given by 
\begin{equation}\label{real tubular zeta}
\zeta_{\mathcal L}(\varepsilon; s):=\frac{\zeta_{\mathcal L}(s)(2\varepsilon)^{1-s}}{s(1-s)},
\end{equation}
and the analog of the above sum in the archimedean tube formula of [\ref{L-vF2}] (as rewritten in 
[\ref{LapPe1}]) is given as in (\ref{sum residue}), except with $\mathcal L_p$ replaced by $\mathcal L$ and with 
$\mathcal D_{\mathcal L}(W)\cup \{0\}$ instead of $\mathcal D_{\mathcal L_p}(W)$. Note that $\zeta_{\mathcal L}(\varepsilon; s)$ typically has a pole at $s=0$, whereas 
$\zeta_{\mathcal L_p}(\varepsilon; s)$ doesn't.
\end{remark}

 \begin{corollary}[$p$-Adic fractal tube formula]\label{ftf}
 If, in addition to the hypotheses in Theorem \ref{dtf}, we assume that all the visible complex dimensions of 
 $\mathcal{L}_p$ are  simple, then  
\begin{equation}\label{pftf} 
V_{\mathcal L_p}(\varepsilon)= \sum_{ \omega \in \mathcal{D}_{\mathcal L_p}(W)} 
c_{\omega} \frac{ \varepsilon^{1-\omega}} {1-\omega } 
+ \mathcal{R}_p (\varepsilon),
\end{equation}
where $c_{\omega}=p^{-1}res\left(\zeta_{\mathcal L_p}; \omega\right)$. 
Here, the error term $\mathcal R_p$ is given by (\ref{error term})
and is estimated by (\ref{estimate}) in the languid case. 
Furthermore, we have $ \mathcal{R}_p (\varepsilon) \equiv 0$ in the strongly languid case  provided we choose $W=\mathbb C$. 
 \end{corollary}
 
 \begin{remark}
 In [\ref{L-vF2}, Ch. 8], under different sets of assumptions, both distributional and pointwise tube formulas are obtained for archimedean fractal strings (and also, for archimedean self-similar fractal strings). 
 (See, in particular, Theorems 8.1 and 8.7, along with \S8.4 in [\ref{L-vF2}].)
 At least for now, in the nonarchimedean case, we limit ourselves to discussing distributional explicit tube formulas. We expect, however, that under appropriate hypotheses, one should be able to obtain a pointwise fractal tube formula for $p$-adic fractal strings and especially, for $p$-adic self-similar strings. 
 In fact, for the simple examples of the nonarchimedean Cantor, Euler and Fibonacci strings, the direct derivation of the fractal tube formula (\ref{pftf})
 yields a formula that is valid pointwise and not just distributionally. (See, in particular, Examples  
  \ref{cantor volume}, \ref{euler volume}, and \ref{fibonacci}.) We leave the consideration of such possible extensions to a future work.
 \end{remark}
 
 \begin{example}[Fractal tube formula for the $p$-adic Euler string]\label{eulerrevisit}
 We now explain how to recover from Theorem \ref{dtf} (or Corollary \ref{ftf}) the tube formula for the Euler string $\mathcal E_p$ obtained via a direct computation in Example \ref{euler volume} of
  \S \ref{inner tube}.
 Indeed, it follows from Corollary \ref{ftf} (applied with $W=\mathbb C$)
  that 
 \begin{equation}\label{etf}
 V_{\mathcal E_p}(\varepsilon)= \frac{1}{p}\sum_{\omega \in \mathcal D_{\mathcal E_p} }
 res(\zeta_{\mathcal E_p}; \omega)
   \frac{\varepsilon^{1-\omega}}{ 1-\omega },
 \end{equation}
 which is exactly the expression obtained for $ V_{\mathcal E_p}(\varepsilon)$ in formula 
 (\ref{VolumeEuler}) of Example \ref{euler volume} since
\[ res(\zeta_{\mathcal E_p}; \omega)=\frac{1}{\log p}\]
  for all $\omega \in \mathcal D_{\mathcal E_p}.$
  (This follows easily from the expression of $\zeta_{\mathcal L_p}$ obtained in Equation
   (\ref{ptheuler}).)
  Note that Corollary \ref{ftf} can be applied here in the strongly languid case when $W=\mathbb C$ and 
  $ \mathcal{R}_p (\varepsilon) \equiv 0$  since, in light of the discussion in \S\ref{Euler string}, all the complex dimensions of $\mathcal E_p$ are simple and $\zeta_{\mathcal E_p}$ is clearly strongly languid of order $\kappa: =0$ and with the constant $A:=p^{-1}.$
  Furthermore, formula (\ref{etf}) can be rewritten in the following more concrete form:
   \begin{equation}\label{cetf}
 V_{\mathcal E_p}(\varepsilon)= \frac{1}{p\log p}\sum_{n\in \mathbb Z}
   \frac{\varepsilon^{1-in\mathbf{p}}}{ 1-in\mathbf{p} },
 \end{equation}
 since
$ \mathcal D_{\mathcal E_p}=\{in\mathbf{p} : n\in \mathbb Z\}$ and $\mathbf{p}=2\pi/\log p$ 
  (as in Equation (\ref{cdes}) of \S\ref{Euler string}).
  
  Finally, note that since the  series 
  \[
  \sum_{n\in \mathbb Z}
   \frac{\varepsilon^{1-in\mathbf{p}}}{ 1-in\mathbf{p} }  \]
   converges pointwise because the associated Fourier series 
   $ \sum_{n\in \mathbb Z}
   \frac{e^{2\pi in x}}{ 1-in\mathbf{p} }$ 
   is pointwise convergent on $\mathbb R$, the $p$-adic fractal tube formulas (\ref{etf})--(\ref{cetf}) actually converge pointwise rather than just distributionally. 
  \end{example}

 \section{Nonarchimedean Self-similar  Strings }\label{sss}
 Nonarchimedean (or $p$-adic) self-similar strings form an important class of $p$-adic fractal strings. In this section, we first recall the construction of these strings, as provided in [\ref{LapLu2}]; see 
 \S\ref{construction}.\footnote{This construction is the nonarchimedean analog of the geometric construction of real (or archimedean) self-similar strings carried out in [\ref{L-vF2}, \S2.1].}  Furthermore, we give an explicit expression for their geometric zeta functions and deduce from it the periodic structure of their poles (or complex dimensions) and zeros, as obtained in 
 [\ref{LapLu2}]; see \S \ref{gzf sss}--\ref{zeros and poles}. 
 Moreover, in \S\ref{exact tf}, we deduce from the results of \S\ref{explicit tf} and \S\ref{gzf sss}--\ref{zeros and poles} the special form of the fractal tube formula for $p$-adic self-similar strings, as obtained in
  [\ref{LapLu3}]. Finally, in \S\ref{amc}, we apply this latter result in order to calculate the average Minkowski content of such strings, as is also done in [\ref{LapLu3}]. 
 
 \subsection{Geometric Construction}\label{construction}
 
 Before explaining how to construct arbitrary $p$-adic self-similar strings, we need to introduce a definition and a few facts pertaining to $p$-adic similarity transformations. 
 
 \begin{definition}\label{similarity mapping}
A map $\Phi: \mathbb Z_p\longrightarrow \mathbb Z_p$ is called a \emph{similarity contraction mapping} of $\mathbb{Z}_p$ if there is a real number $r \in (0,1)$ such that 
\[
|\Phi(x)-\Phi(y)|_p=r\cdot|x-y|_p,
\]
for all $x, y \in \mathbb Z_p$. 
\end{definition}
 
Unlike in Euclidean space (and in the real line $\mathbb R$, in particular), it is not true that   every similarity transformation of $\mathbb Q_p$ (or of $\mathbb Z_p$) is necessarily affine.
Actually, in the nonarchimedean world (for example, in $\mathbb Q_p^d$, with $d\geq 1$), and in the $p$-adic line $\mathbb Q_p$, in particular, there are a lot of similarities which are not affine. 
However, it is known (see, e.g., [\ref{Sch}]) that every analytic similarity must be affine.\footnote{Here, a map $f:\mathbb Q_p \longrightarrow \mathbb Q_p$ is said to be analytic if it admits a convergent power series expansion about 0, and with coefficients in $\mathbb Q_p$, that is convergent in all of $\mathbb Q_p$.}
    Hence, from now, we are working with a similarity contraction mapping 
     $\Phi: \mathbb Z_p\longrightarrow \mathbb Z_p$ that is affine. 
     Thus we assume that there exist constants $a,b\in \mathbb Z_p$ with $|a|_p<1$ such that    
    $\Phi(x)=ax+b$ for all $x\in \mathbb Z_p$. Regarding the scaling factor $a$ of the contraction, it is well known that it can be written as $a=u\cdot p^n,$ for some unit $u\in \mathbb{Z}_p$ (i.e., $|u|_p=1$) and $n\in \mathbb N^{*}$ (see [\ref{Neu}]).  Then $r=|a|_p=p^{-n}.$  We summarize this fact in the following lemma:
 
  \begin{lemma}\label{scaling} 
  Let $\Phi(x)=ax+b$ be an affine similarity contraction mapping of $\mathbb {Z}_p$ with the scaling ratio $r$. Then $b\in \mathbb Z_p$ and  $a\in p\mathbb Z_p,$ and the scaling factor is $r=|a|_p=p^{-n}$ for some $n\in \mathbb N^{*}$. 
     \end{lemma}
 
 \begin{figure}[ht]\label{p-adic self-similar construction}
\psfrag{dots}{$\vdots$}
\psfrag{Zp}{$\mathbb Z_p$}
\psfrag{P1}{$\Phi_1(\mathbb Z_p)$}
\psfrag{dots}{$\cdots$}
\psfrag{PJ}{$\Phi_N(\mathbb Z_p)$}
\psfrag{G1}{$G_1$}
\psfrag{GK}{$G_K$}
\psfrag{P11}{$\Phi_{11}(\mathbb Z_p)$}
\psfrag{P1J}{$\Phi_{1N}(\mathbb Z_p)$}
\psfrag{G11}{$\Phi_1(G_1)$}
\psfrag{G1K}{$\Phi_1(G_K)$}
\psfrag{PJ1}{$\Phi_{N1}(\mathbb Z_p)$}
\psfrag{PJJ}{$\Phi_{NN} \mathbb Z_p $}
\psfrag{GJ1}{$\Phi_N(G_1)$}
\psfrag{GJK}{$\Phi_N(G_K)$}
\psfrag{vdots}{$\vdots$}
\raisebox{-1cm}
{\psfig{figure=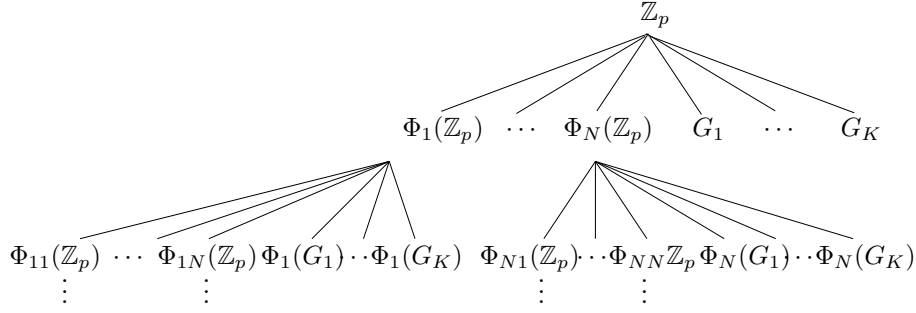, height=4cm}}
\caption{Construction of a  $p$-adic self-similar string.} \end{figure}

For simplicity, let us take the unit interval (or ball) $\mathbb{Z}_p$ in $\mathbb{Q}_p$ and construct a 
\emph{ $p$-adic} (or\emph{ nonarchimedean}) \emph{self-similar  string} $\mathcal{L}_p$ as follows (see [\ref{LapLu2}]).\footnote{In the sequel, $\mathcal L_p$ is interchangeably called a \emph{$p$-adic} or \emph{nonarchimedean} self-similar string.}
  Let $N\geq2$ be an integer and  $\Phi_1, \dots, \Phi_N: \mathbb Z_p\longrightarrow \mathbb Z_p$ be $N$ affine similarity contraction mappings with the respective scaling ratios $r_1, \dots, r_N \in (0,1)$ satisfying 
\begin{equation}\label{ratios}
1>r_1\geq r_2\geq \cdots \geq r_N>0;
\end{equation}
see Figure 9.  Assume that 
\begin{equation}\label{rho}
\sum_{j=1}^N r_j <1,
\end{equation}
and the images $\Phi_j(\mathbb{Z}_p)$ of $\mathbb{Z}_p$ do not overlap, i.e., 
$\Phi_j(\mathbb{Z}_p) \cap \Phi_l(\mathbb{Z}_p)= \emptyset$
for all $j\neq l$. 
Note that it follows from Equation (\ref{rho}) that $\bigcup_{j=1}^N \Phi_j(\mathbb Z_p)$ is not all of $\mathbb Z_p$.
We therefore have the following (nontrivial) decomposition of
$\mathbb Z_p$ into disjoint $p$-adic intervals:
\begin{equation}\label{first generation}         
 \mathbb Z_p = \bigcup_{j=1}^N \Phi_j(\mathbb Z_p)\cup \bigcup_{k=1}^K G_k,
\end{equation}
where $G_k$ is defined below. 

In a procedure reminiscent of the construction of the ternary Cantor set in \S\ref{C} or of the 3-adic Cantor set in \S\ref{C3}, we then subdivide the interval $\mathbb{Z}_p$ by means of the subintervals $\Phi_j(\mathbb{Z}_p)$.
Then the convex \footnote{We choose the convex components instead of the connected components because $\mathbb{Z}_p$ is totally disconnected. Naturally, no such distinction is necessary in the archimedean case;  cf. [\ref{L-vF2}, \S2.1.1]. Here and elsewhere in this paper, a subset $E$ of $\mathbb Q_p$ is said to be `convex' if for every $x,y\in E$, the $p$-adic segment $\{tx+(1-t)y: t\in \mathbb Z_p\}$ lies entirely in $E$.} 
components of 
\[
\mathbb Z_p \backslash \bigcup_{j=1}^N \Phi_j(\mathbb Z_p)
\]
are the first \emph{substrings} of the $p$-adic  self-similal string $\mathcal L_p$, say $G_1, G_2, \ldots, G_K,$ with $K\geq 1$.  These intervals $G_k$ are called the \emph{generators},  the deleted intervals in the first generation of the construction of $\mathcal L_p.$\footnote{Their archimedean counterparts are called `gaps' in  [\ref{L-vF2}, Ch. 2 and \S8.4], where archimedean self-similar strings are introduced.}
  The length of each $G_k$ is denoted by $g_k$; so that $g_k=\mu_H(G_k)$.\footnote{We note that the lengths $g_k$ ($k=1,2,\ldots, K$) will sometimes be called the (nonarchimedean) `gaps' or `gap sizes' in the sequel.}
Without loss of generality, we may assume that the lengths
$g_1, g_2, \dots, g_K$ of the first substrings (i.e., intervals) of $\mathcal L_p$ satisfy
\begin{equation}\label{gaps}
1>g_1\geq g_2\geq \cdots \geq g_K >0. 
\end{equation}
It follows from Equation (\ref{first generation}) and the additivity of Haar measure $\mu_H$ that 
\begin{equation}\label{gap identity}
\sum_{j=1}^N r_j + \sum_{k=1}^K g_k =1.
\end{equation}
We then repeat this process with each of the remaining subintervals 
$\Phi_j(\mathbb{Z}_p)$ of $\mathbb Z_p,$ for $j=1,2,\ldots,N$. 
And so on, ad infinitum. 
As a result, we obtain a \emph{$p$-adic  self-similar  string} $\mathcal{L}_p=l_1, l_2, l_3, \dots,$
consisting of intervals of length $l_n$ given by 
\begin{equation}\label{length}
r_{\nu_1}r_{\nu_2}\cdots r_{\nu_q}g_k, 
\end{equation}
for $k=1, \ldots, K$ and all choices of $q\in \mathbb N$ and $\nu_1, \ldots, \nu_q \in \{1, \ldots, N\}$.
Thus, the lengths are of the form 
$r_1^{e_1}\dots r_N^{e_N}g_k$ with $e_1, \dots, e_N \in \mathbb{N}$ (but not all zero).

In [\ref{LapLu2}], the classic notion of self-similarity is extended to the nonarchimedean setting, much as in [\ref{Hut}], where the underlying complete metric space is allowed to be arbitrary. We note that the next result follows by applying the classic Contraction Mapping Principle to the complete metric space of all nonempty compact subsets of $\mathbb Z_p.$\footnote{Recall that $\mathbb Z_p$ is  complete since it is a compact metric space (see \S\ref{C3} ).}

 \begin{theorem}[$p$-Adic self-similar set] \label{Hutchinson}
There is a unique nonempty compact subset $\mathcal S_p$ of $\mathbb Z_p$ such that 
\[
\mathcal S_p=\bigcup_{j=1}^N \Phi_j(\mathcal S_p).
\]
The set $\mathcal S_p$ is called the $p$-adic self-similar set associated with the self-similar system 
$\mathbf\Phi=\{\Phi_1, \ldots, \Phi_N\}.$ (It is also called the $\mathbf\Phi$-invariant set.)
\end{theorem}

The relationship between the $p$-adic self-similar string $\mathcal L_p$ and the above $p$-adic self-similar set $\mathcal S_p$ is given by the following theorem, also obtained in [\ref{LapLu2}]:\footnote{In Theorem \ref{complementary}, $\mathcal L_p$ is not viewed as a sequence of lengths but is viewed instead as the open set which is canonically given by a disjoint union of intervals (its $p$-adic convex components), as described in the above construction of a $p$-adic self-similar string.}

\begin{theorem}\label{complementary}
 (i)  $\mathcal L_p=\mathbb Z_p \backslash \mathcal S_p,$
the complement of $\mathcal S_p$ in $\mathbb Z_p$. 

(ii) 
$\mathcal{L}_p=\bigcup_{\alpha=0}^{\infty}\bigcup_{w\in W_{\alpha-1}}\bigcup_{k=1}^K\Phi_w(G_k),$
while
$\mathcal S_p=\bigcap_{\alpha=0}^{\infty}\bigcup_{w\in W_{\alpha}} \Phi_w(\mathbb{Z}_p),$
where $W_{\alpha}=\{1,2, \ldots, N\}^{\alpha}$ denotes the set of all finite words on $N$ symbols, of length $\alpha,$ and
 $\Phi_w:=\Phi_{w_{\alpha}} \circ \cdots  \circ \Phi_{w_1}$ 
 for~ $w=(w_1, \dots, w_{\alpha})\in W_{\alpha}. \footnote{By convention, $\Phi_w(G_k)=\emptyset$ if $w\in W_{-1}$.}$
\end{theorem}

 \begin{figure}[h]
\psfrag{dots}{$\vdots$}
\psfrag{Z3}{$\mathbb Z_3$}
\psfrag{Z0}{$0+\mathbb Z_3$}
\psfrag{Z1}{$1+3\mathbb Z_3=G$}
\psfrag{Z2}{$2+\mathbb Z_3$}
\psfrag{1}{$0+9\mathbb Z_3$}
\psfrag{2}{$\Phi_1(G)$}
\psfrag{3}{$6+9\mathbb Z_3$}
\psfrag{4}{$2+9\mathbb Z_3$}
\psfrag{5}{$\Phi_2(G)$}
\psfrag{6}{$8+9\mathbb Z_3$}
\raisebox{-1cm}
{\psfig{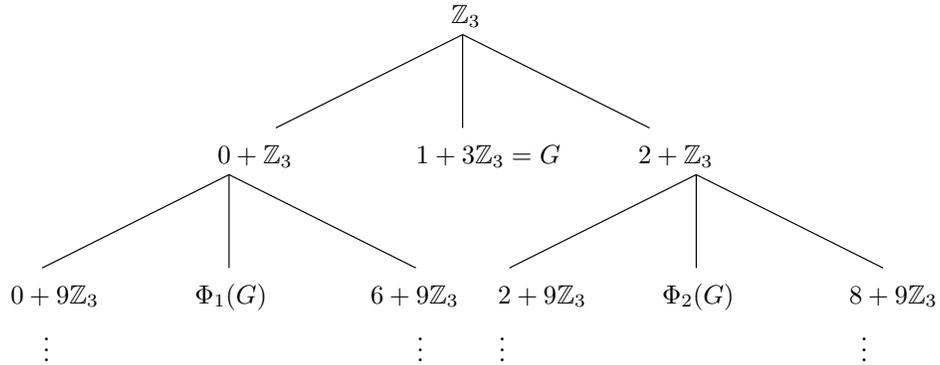}}
\caption{Construction of the nonarchimedean Cantor string $\mathcal{CS}_3$ via an IFS.}
\end{figure}
\begin{example}[Nonarchimedean Cantor string as a 3-adic self-similar string]\label{cantor}
In this example, we review the construction (given in \S\ref{CS3}) of the nonarchimedean Cantor string 
$\mathcal{CS}_3$, as introduced in [\ref{LapLu1}] and revisited in [\ref{LapLu2}]. Our main point here is to stress the fact that $\mathcal{CS}_3$ is a special case of a $p$-adic self-similar string, as constructed just above, and to prepare the reader for more general results about nonarchimedean self-similar strings, as given in the rest of \S\ref{sss}. 

Let $\Phi_1, \Phi_2: \mathbb Z_3\longrightarrow \mathbb Z_3$ be the two affine similarity contraction mappings of $\mathbb {Z}_3$ given by 
\begin{equation}\label{maps}
\Phi_1(x)=3x \quad \mbox{and}\quad \Phi_2(x)=2 + 3x,
\end{equation}
with the same scaling ratio $r=1/3$ (i.e., $r_1=r_2=1/3$).
By analogy with the construction of the real Cantor string (see \S\ref{CS}), subdivide the interval $\mathbb{Z}_3$ into subintervals 
\[
\Phi_1(\mathbb{Z}_3)=0+3\mathbb Z_3 \quad \mbox{and} \quad
 \Phi_2(\mathbb{Z}_3)=2+3\mathbb Z_3.\]
The remaining (3-adic) convex component
\[
\mathbb Z_3 \backslash \bigcup_{j=1}^2 \Phi_j(\mathbb Z_3)=1+3\mathbb Z_3=G
\]
is the first substring of a  $3$-adic self-similar string, called the \emph{nonarchimedean} 
 \emph{Cantor string} and denoted by $\mathcal{CS}_3$ [\ref{LapLu1}]. (Of course, this is the same $p$-adic fractal string $\mathcal{CS}_3$ as the one constructed in \S\ref{CS3}.)
The length of $G$ is
 $l_1=\mu_H(1+3\mathbb{Z}_p)=3^{-1}.$
By repeating this process with the remaining subintervals 
$\Phi_j(\mathbb{Z}_3),~ \mbox{for} ~ j=1,2$, and continuing on, ad infinitum, we eventually obtain a sequence 
$\mathcal{CS}_3 =l_1, l_2, l_3, \dots,$
associated with the open set resulting from this construction and consisting of intervals of lengths $l_v=3^{-v}$ with multiplicities $m_v=2^{v-1}$, for $v\in \mathbb N^*$. 
As we have seen in \S\ref{CS3}, and as follows from this construction (see Figure 10 and Equation (\ref{maps}), along with part (ii) of Theorem \ref{complementary}), the nonarchimedean Cantor string  $\mathcal{CS}_3$ can also be written as 
\begin{equation}\label{CS_3}
\mathcal{CS}_3=(1+3\mathbb{Z}_3) \cup  (3+9\mathbb{Z}_3) \cup  (5+9\mathbb{Z}_3) \cup \cdots.
\end{equation}
 By definition, the geometric zeta function of $\mathcal{CS}_3$ is given by
  \begin{eqnarray*}
\zeta_{\mathcal{CS}_3}(s)&=&(\mu_H(1+3\mathbb{Z}_3))^s + (\mu_H(3+9\mathbb{Z}))^s + (\mu_H(5+9\mathbb{Z}_3))^s + \cdots\\
&=&\sum_{v=1}^{\infty}\frac{2^{v-1}}{ 3^{vs}} 
=\frac{3^{-s}}{1-2\cdot 3^{-s}} \quad \mbox{for} \quad \Re(s)>\log_32.
\end{eqnarray*}
Hence, by analytic continuation, the meromorphic extension of $\zeta_{\mathcal{CS}_3}$ to the entire complex plane  
$\mathbb C$ exists and  is given by
\begin{equation}\label{cantor zeta}
\zeta_{\mathcal{CS}_3}(s)=\frac{3^{-s}}{1-2\cdot 3^{-s}}, \quad  \mbox{for} ~s\in \mathbb C,
\end{equation}
with poles at 
\[\omega=\frac{\log2}{\log3} +i n \frac{ 2 \pi}{\log3}, \quad n \in \mathbb Z.\] 
 Therefore, we recover the fact that the set of complex dimensions  of $\mathcal {CS}_3$ is given by
\begin{equation}\label{cd cantor} 
\mathcal{D}_{\mathcal {CS}_3}=\{   D+i n \mathbf{p}~|~n \in \mathbb{Z} \},
\end{equation}
where $D=\log_32$ is the dimension of $\mathcal {CS}_3$ and $\mathbf{p}=2\pi / \log3$ is its oscillatory period.  (This terminology will be explained in \S\ref{Periodicity} and \S\ref{zeros and poles}.)
Naturally, Equations (\ref{cantor zeta}) and (\ref{cd cantor}) are in agreement with the statement of Theorem \ref{zcd}.
Finally, note that $\zeta_{\mathcal{CS}_3}$ is a rational function of $z:=3^{-s}$, i.e.,
\begin{equation*}
\zeta_{\mathcal{CS}_3}(s)=\frac{z}{1-2z}. 
\end{equation*}
We refer the interested reader to [\ref{LapLu1}]  as well as to \S\ref{C3}--\ref{comparison} above for additional information concerning the nonarchimedean Cantor string $\mathcal{CS}_3$ and the associated nonarchimedean Cantor set $\mathcal C_3$. 
We just mention here that in light of part (i) of Theorem \ref{complementary}, we can recover the 3-adic Cantor set $\mathcal C_3$ as the complement of the 3-adic Cantor string $\mathcal {CS}_3$ in the unit interval (and vice-versa):
\begin{equation}
\mathcal{CS}_3=\mathbb Z_3 \backslash \mathcal C_3, \quad \mbox{and so} \quad 
\mathcal C_3=\mathbb Z_3 \backslash \mathcal{CS}_3.
\end{equation}
Indeed, according to Theorem \ref{Hutchinson} (and in agreement with Theorem \ref{invariant} of \S\ref{C3}), $\mathcal C_3$ is the self-similar set associated with the IFS $\mathbf{\Phi}=\{\Phi_1, \Phi_2\}$.
\end{example}

\begin{example}[The explicit tube formula for the nonarchimedean Cantor string]\label{cantor volume}
In this example, we explain how to derive the exact fractal tube formula for $\mathcal{CS}_3$ as stated in Equation (\ref{tfCS3}), in two different ways:

(i) First, via a direct computation (much as we derived the tube formula for the $p$-adic Euler string in Example \ref{euler volume}). 

(ii) Second, as a special case of the general $p$-adic tube formula obtained in \S\ref{explicit tf}.

Let $\varepsilon >0$.  Then, by Theorem \ref{thin}, we have 
\begin{equation}
V_{\mathcal {CS}_3}(\varepsilon)=\frac{1}{3}\sum_{n=k+1}^{\infty}
\frac{2^{n-1}}{3^n}
=\frac{1}{3}\left( \frac{2}{3}\right)^k,
\end{equation}
where $k:=[\log_3\varepsilon^{-1}]$. 
Let $x:=\log_3\varepsilon^{-1}=k+\{x\}$, where $\{x\}$ is the fractional part of $x$. Then a simple computation shows that
$\left(\frac{2}{3}\right)^x=\varepsilon^{1-D} ~\mbox{and}~ e^{2\pi i n x}=\varepsilon^{-in\mathbf{p}}$, 
with $D=\log_32$ and $\mathbf{p}=2\pi /\log 3$ as in Example \ref{cantor}.
Using the Fourier expansion for $b^{-\{x\}}$, as given by Equation (\ref{fourier}), for $b=3^{-1}$
and the above value of $x$, we obtain an expansion in terms of the complex dimensions $\omega=D+in\mathbf{p}$ of $\mathcal{CS}_3$:

\begin{eqnarray}
V_{\mathcal{CS}_3}(\varepsilon)\nonumber
&=& \frac{3^{-1}}{2\log 3}\sum_{n\in \mathbb Z}\frac{\varepsilon^{1-D-in\textbf{p}}} {1-D-in\textbf{p}}\\
&=& \frac{3^{-1}}{ 2\log 3} \sum_{\omega \in \mathcal D_{\mathcal{CS}_3}}\frac{\varepsilon^{1-\omega  }}{1-\omega},\label{tfc2}
\end{eqnarray}
since $\mathcal D_{\mathcal{CS}_3}$ is given by (\ref{cd cantor}).
Next, using Equation (\ref{cantor zeta}), we see that
 \[res(\zeta_{\mathcal{CS}_3}; \omega)=\frac{1}{2\log3},\]
 independently of $\omega \in \mathcal{D}_{\mathcal{CS}_3 }$, and so the exact fractal tube formula for the nonarchimedean Cantor string is found to be 
\begin{equation}\label{tfc}
V_{\mathcal{CS}_3}(\varepsilon)=
\frac{1}{ 3} \sum_{\omega \in \mathcal D_{\mathcal {CS}_3}}res(\zeta_{\mathcal{CS}_3}; \omega)\frac{\varepsilon^{1-\omega  }}{1-\omega}.
\end{equation}
Note that since $\mathcal{CS}_3$ has simple complex dimensions, we may also apply Corollary \ref{ftf} (in the strongly languid case when $W=\mathbb C$) in order  to precisely recover Equation (\ref{tfc}). (Alternatively, we could use Theorem \ref{5.13} in \S\ref{exact tf} below.)

We may rewrite (\ref{tfc2}) or (\ref{tfc}) in the following form (which agrees with the tube formula to be obtained in Theorem \ref{5.13}):
\[
V_{\mathcal{CS}_3}(\varepsilon)=\varepsilon^{1-D}G_{\mathcal{CS}_3}(\log_{3}\varepsilon^{-1}),
\]
where (as in Equation (\ref{4}) of \S\ref{comparison}) $G_{\mathcal{CS}_3}$ is the nonconstant periodic function (of period 1) on $\mathbb R$ given by 
\[
G_{\mathcal{CS}_3}(x):=\frac{1}{6\log3}\sum_{n\in \mathbb Z}\frac{e^{2\pi inx}}{1-D-in\mathbf{p}}.
\]
Finally, we note that since the Fourier series 
\[
\sum_{n\in \mathbb Z}\frac{e^{2\pi inx}}{1-D-in\mathbf{p}}
\]
is pointwise convergent on $\mathbb R$, the above direct computation of $V_{\mathcal{CS}_3}(\varepsilon)$
shows that (\ref{tfc2}) and (\ref{tfc}) actually hold pointwise rather than distributionally.
\end{example}

\subsection{Geometric Zeta Function of  $p$-Adic Self-Similar  Strings}\label{gzf sss}
  In this section, as well as in \S\ref{Periodicity} and \S\ref{zeros and poles}, we will survey results obtained in [\ref{LapLu2}] about the geometric zeta functions and the complex dimensions of $p$-adic self-similar strings. 
  
  In the next theorem, we provide a first expression for the geometric zeta function of a nonarchimedean self-similar string. At first sight, this expression is almost identical to the one obtained in the archimedean case in [\ref{L-vF2}, Thm. 2.4]. 
  Later on, however, we will see that unlike in the archimedean case where the situation is considerably more subtle and complicated (cf. [\ref{L-vF2}, Thms. 2.17 and 3.6]), this expression can be significantly  simplified since the two potentially transcendental functions appearing in the denominator and numerator of Equation (\ref{sszf}) below can always be made rational; see Theorem 
  \ref{rationality} in \S\ref{Periodicity}.

\begin{theorem}\label{geometric zeta function}
Let $\mathcal{L}_p$ be a  $p$-adic self-similar  string with scaling ratios $\{r_j\}_{j=1}^N$ and gaps $\{g_k\}_{k=1}^K$, as in the above construction. 
Then the geometric zeta function of $\mathcal{L}_p$
has a meromorphic extension to the whole complex plane $\mathbb{C}$ and is given by
\begin{equation}\label{sszf}
\zeta_{\mathcal{L}_p}(s)=
\frac
{\sum_{k =1}^{K} g_{k}^s}
{1-\sum_{j=1}^N r_j ^s}, 
\quad  \mbox{for} \quad  s\in \mathbb{C}.
\end{equation}
\end{theorem}
\begin{corollary}
The set of complex dimensions of a $p$-adic  self-similar  string $\mathcal{L}_p$ is contained in the set of complex solutions $\omega$ of the Moran equation $\sum_{j=1}^N r_j ^{\omega}=1.$ If the string has a single generator (i.e., if $K=1$), then this inclusion is an equality.\footnote{See Examples \ref{cantor}, \ref{fibonacci} and Theorem \ref{rationality}.}
\end{corollary} 

\begin{definition}\label{latticenonlattice} 
A $p$-adic  self-similar  string $\mathcal{L}_p$ is said to be lattice (or nonlattice) if the multiplicative group generated by the  scaling ratios 
$r_1, r_2, \ldots, r_N$ is discrete (or dense)  in $(0,\infty)$.\footnote{More precisely, much as in the archimedean case in [\ref{L-vF2}, Ch. 3], we say that $\mathcal L_p$ is \emph{lattice} if the multiplicative group generated by the \emph{distinct} scaling ratios of $\mathcal L_p$ is a subgroup of $(0,\infty)$ of rank one (i.e., if it is a free abelian group). As it turns out, in the present case of $p$-adic self-similar strings, it does not matter whether we use this slightly refined definition or the one given in Definition \ref{latticenonlattice}; see Remark \ref{remark 4.11}. } 
\end{definition}
 
 \begin{theorem}\label{lattice}
Every  $p$-adic self-similar  string is lattice.  
\end{theorem}

 \begin{remark}\label{remark 4.11} 
 Theorem \ref{lattice} follows from the fact that all the scaling ratios $r_j$ must belong to the multiplicative group $p^ {\mathbb Z }.$ In fact, much more is true since the gaps $g_k$ must also belong to  $p^{ \mathbb Z },$ as will be discussed below in more detail in \S\ref{Periodicity}. It follows that  $p$-adic self-similar  strings are lattice strings in a very strong sense, namely, their geometric zeta functions are \emph{rational functions} of a suitable variable $z$ (see Theorem \ref{rationality} below).  \end{remark}

\begin{remark}Theorem \ref{lattice} is in sharp contrast with the usual theory of real (or archimedean) self-similar strings developed in [\ref{L-vF2}, Chs. 2 and 3].  Indeed,  there are both lattice and nonlattice strings  in the archimedean case. Furthermore, generically, archimedean self-similar strings are nonlattice. Moreover, it is shown in [\ref{L-vF2}, Ch. 3] by using Diophantine approximation that every nonlattice string in $\mathbb R=\mathbb Q_{\infty}$ can be approximated by a sequence of lattice strings with oscillatory periods increasing to infinity. It follows that the complex dimensions of an archimedean nonlattice string are quasiperiodically distributed (in a very precise sense, that is explained in \emph{loc. cit.}) because the complex dimensions of archimedean lattice strings are periodically distributed along finitely many vertical lines. Clearly, there is nothing of this kind in the nonarchimedean case since $p$-adic self-similar strings are necessarily lattice. 
\end{remark}

\begin{remark} The $p$-adic Euler string $\mathcal E_p$, discussed in \S\ref{Euler string}, is \emph{not} self-similar because  $\mathcal E_p$ has dimension
$D=0,$ whereas the requirement that $N\geq2$ in the definition of a $p$-adic self-similar string implies that 
$D>0$ for any  $p$-adic  self-similar string.
\end{remark}

\subsection{$p$-Adic Self-Similar Strings Are Strongly Lattice}\label{Periodicity}
 A small modification of the above argument enables us to show that every $p$-adic self-similar  string is lattice in a much stronger sense, as we now explain. It will follow (see Theorem \ref{periodicity}) that \emph{not only the poles} (i.e., the complex dimensions of $\mathcal{L}_p$) \emph{but also the zeros of $\zeta_{\mathcal{L}_p}$} \emph{are periodically distributed.} Accordingly, we will say that   $p$-adic self-similar strings are \emph{strongly lattice}.

We introduce some necessary notation. First, by Lemma \ref{scaling}, we write
\begin{equation*}\label{r}
r_j=p^{-n_j}, \quad \mbox{with} \quad n_j\in \mathbb N^* \quad  \mbox{for} \quad j=1, 2, \ldots, N.
\end{equation*}
Second, we write 
\begin{equation*}
g_k=\mu_H(G_k)=p^{-m_k}, \quad \mbox{with} \quad m_k\in \mathbb N^* \quad \mbox{for} \quad k=1, 2, \ldots, K.
\end{equation*}
Third, let 
\begin{equation*}
d= \mbox{gcd}\{n_1, \ldots, n_N, m_1, \ldots, m_K\}.
\end{equation*}
Then there exist positive integers $n_j' ~\mbox{and}~ m_k'$ such that
\begin{equation}\label{n'}
n_j=dn_j'  \quad \mbox{and} \quad m_k=dm_k' \quad \mbox{for}  \quad j=1, \ldots, N  \quad \mbox{and} \quad  k=1,\ldots, K.
\end{equation}
Finally, we set \footnote{Note that by construction, $r_j=r^{n_j'}$ and $g_k=r^{m_k'}$ for $j=1, \dots, N$ and $k=1,\dots, K.$ Hence, $r$ is the multiplicative generator in $(0,1)$ of the rank one group generated by $\{r_1, \ldots, r_N, g_1, \dots, g_K\}$ (or, equivalently, by either
 $\{r_1, \ldots, r_N\}$ or $\{g_1, \ldots, g_K\}$).}
 
\begin{equation}\label{Q}
p^d=1/r.
\end{equation}
Without loss of generality, we may assume that the scaling ratios $r_j$ and  the gaps $g_k$ are written in nonincreasing order as in Equations (\ref{ratios}) and (\ref{gaps}), respectively; so that 
\begin{equation}
0<n_1'\leq n_2' \leq \cdots \leq n_N'  \quad \mbox{and} \quad 0<m_1'\leq m_2' \leq \cdots \leq m_K'. 
\end{equation}

\begin{theorem}\label{rationality}
Let $\mathcal{L}_p$ be a $p$-adic self-similar  string and $z=r^s,$ with $r=p^{-d}$ as in Equation (\ref{Q}).
Then the geometric zeta function
 $\zeta_{\mathcal{L}_p}$ of $\mathcal{L}_p$ is a rational function in $z$. Specifically, 
 \begin{equation}
\zeta_{\mathcal{L}_p}(s)=
\frac
{\sum_{k =1}^{K} z^{m_{k}'}}
{1-\sum_{j=1}^N z ^{n_j'}},
\end{equation}
where $m_k', n_j'\in \mathbb N^*$ are given by Equation (\ref{n'}).
\end{theorem}

 \begin{definition}\label{oscillatoryperiod} 
 Let $\mathbf{p}=\frac{2\pi}{d\log p}$. Then $\mathbf{p}$ is called the \emph{oscillatory period} of $\mathcal{L}_p.$ \end{definition}

\subsection{Periodicity of the Poles and the Zeros of $\zeta_{\mathcal L_p}$}\label{zeros and poles}

The following result (also from [\ref{LapLu2}]) is the nonarchimedean counterpart of [\ref{L-vF2}, Thms.  2.17 and 3.6], which provide the rather subtle structure of the complex dimensions of archimedean self-similar strings. It is significantly simpler, however, due in part to the fact that nonlattice $p$-adic self-similar strings do not exist.  

To avoid any confusion, we stress that in the statement of the next theorem, $\zeta_{\mathcal L_p}$ is viewed as a function of the original complex variable $s$. Moreover,  it follows from Theorem \ref{D=sigma} in \S\ref{Mdimension} above and from Theorem \ref{periodicity} below that $D$, 
 the dimension of $\mathcal L_p$, defined as the \emph{abscissa of convergence} of the Dirichlet series originally defining $\zeta_{\mathcal L_p}$ (and sometimes also denoted by $\sigma=\sigma_{\mathcal L_p} $ here) coincides with $\delta$ and the \emph{Minkowski dimension} $D_M=D_{M,\mathcal L_p}$ of $\mathcal L_p$: 
 \[D=D_M=\sigma=\delta,\]
  where $\delta$ is the \emph{similarity dimension} of $\mathcal L_p$, i.e., the unique \emph{real} solution of the Moran equation $\sum_{j=1}^Nr_j^s=1$. Hence, in the present case of $p$-adic self-similar  strings, there is no need to distinguish between these various notions of `fractal dimensions'. (See Remark \ref{D<1} below for more information.)

\begin{theorem}[Structure of the complex dimensions]\label{periodicity} 
Let $\mathcal{L}_p$ be a nontrivial $p$-adic self-similar   string. Then \\
\indent(i) The complex dimensions of $\mathcal{L}_p$ and the zeros of $\zeta_{\mathcal{L}_p}$ are periodically distributed along finitely many  vertical lines, with period $\mathbf p$, the oscillatory period of $\mathcal{L}_p$ (as given in Definition \ref{oscillatoryperiod}). \\
\indent(ii) Furthermore, along a given vertical line, each pole (respectively, each zero) of $\zeta_{\mathcal{L}_p}$ has the same multiplicity. \\
\indent(iii) Finally, the dimension $D$ of $\mathcal{L}_p$ is the only complex dimension that is located on the real axis.\footnote{By contrast, it is immediate to check that there are no real zeros (still in the $s$ variable).} Moreover, $D$ is simple \footnote{i.e., $D$ is a simple pole of $\zeta_{\mathcal{L}_p}.$} and is located on the right most vertical line. That is, $D$ is equal to the maximum of the real parts of the complex dimensions.  
\end{theorem}

\begin{remark} 
As will be apparent to the expert reader, the situation described above---specifically, the rationality of the zeta function in the variable $z=r^{s}$, with $r=p^{-d}$, and the ensuing periodicity of the poles and the zeros---is analogous to the one encountered for a curve (or more generally, a variety) over a finite field $\mathbb{F}_{p^d}$; see, e.g., Chapter 3 of [\ref{ParSh}]. In this analogy, the prime number $p$ naturally corresponds to the characteristic of the finite field, and $p^d=r^{-1}$ is the analog of the cardinality of the field. 
\end{remark}

\begin{remark} \label{D<1}
By Theorem \ref{D=sigma}, $D$ is also the inner Minkowski dimension $D_M$ of the self-similar set associated with the present self-similar system $\mathbf{\Phi}$:
 $D=D_M=\sigma,$
the abscissa of convergence of the geometric zeta function. Moreover, in light of (\ref{sszf}) and part (iii) of Theorem \ref{periodicity} above, it always coincides with the similarity dimension $\delta$ of 
 $\mathbf{\Phi}$. Namely, $D$ is the unique real solution of the Moran equation
 $\sum_{j=1}^N r_j^s=1$ (as in [\ref{Mor}]): 
\begin{equation}\label{Moran}
\sum_{j=1}^N r_j^{D}=1.
\end{equation}
As a result, $D=\sigma=D_M=\delta<1,$ since by assumption (see Equation (\ref{rho}) above), $\sum_{j=1}^N r_j<1$. This last observation will enable us, in particular, to apply the fractal tube formula (Theorem \ref{dtf} and Corollary \ref{ftf}) to any $p$-adic self-similar  string. 
\end{remark}

We next supplement the above results by establishing a theorem obtained in [\ref{LapLu3}] and which  will be very useful to us in \S\ref{exact tf} in order to simplify the tube formula associated with a $p$-adic self-similar string. 

According to part (i) of Theorem \ref{periodicity}, there exist finitely many poles 
\[\omega_1,\ldots,  \omega_q,\]
with $\omega_1=D$ and $\Re(\omega_q)\leq \cdots \leq \Re(\omega_2)<D$,
such that 
\[\mathcal D_{\mathcal L_p}=\{\omega_u+in\mathbf{p} ~|~ n\in \mathbb Z, ~u=1, \ldots, q\}.\]
Furthermore, each complex dimension $D+in\mathbf{p}$ is simple (by parts (ii) and (iii) of Theorem \ref{periodicity}) and the residue of $\zeta_{\mathcal L_p}(s)$ at $s=D+in\mathbf{p}$ is independent of $n\in \mathbb Z$ and equal to 
\begin{equation}\label{residue formula}
res(\zeta_{\mathcal{L}_p}; D+in\mathbf{p})=
\frac
{\sum_{k =1}^{K} r^{m_{k}'D}}
{\log{r^{-1}}\sum_{j=1}^N n_j'r^{n_j'D}}.
\end{equation}
The latter fact (concerning residues) is an immediate consequence of the following   
result from [\ref{LapLu3}]:
 
\begin{theorem}\label{residue}
(i) For each $v=1, \ldots, q,$ the principal part of the Laurent series of $\zeta_{\mathcal L_p}(s)$ at 
$s=\omega_v+in\mathbf{p}$ does not depend on $n\in \mathbb Z$. \\
(ii) Moreover, let $u\in \{1, \ldots, q\}$ be such that $\omega_u$ (and hence also $\omega_u+in\mathbf{p}$, for every $n\in \mathbb Z$, by part (ii) of Theorem \ref{periodicity} ) is simple. 
Then the residue of $\zeta_{\mathcal L_p}(s)$ at $s=\omega_u+in\mathbf{p}$ is independent of  $n\in \mathbb Z$ and 
\begin{equation}\label{residue equation}
res(\zeta_{\mathcal{L}_p}; \omega_u+in\mathbf{p})=
\frac
{\sum_{k =1}^{K} r^{m_{k}'\omega_u}}
{\log{r^{-1}}\sum_{j=1}^N n_j'r^{n_j'\omega_u}}.
\end{equation}
In particular, this is the case for $\omega_1=D.$ 
\end{theorem} 
 
 Note that by contrast, in the lattice case of the archimedean theory of self-similar strings developed in [\ref{L-vF2}, Chs. 2 and 3], we had to assume that the gap sizes (and not just the scaling ratios) are integral powers of $r$ in order to obtain the counterpart of Theorem \ref{residue}. 
 
\begin{remark}[Comparison with the archimedean case]\label{stronglattice}
Part (i) of Theorem \ref{periodicity}, along with Theorem \ref{rationality}, shows that in some very explicit sense, the theory of  $p$-adic self-similar  strings is both simpler and more natural than its archimedean counterpart. Indeed, not only is it the case that every $p$-adic self-similar  string $\mathcal L_p$ is lattice, but both the zeros and poles of $\zeta_{\mathcal L_p}(s)$ are periodically distributed along vertical lines, with the same period. By contrast, even if an archimedean self-similar string $\mathcal L$ is assumed to be `lattice', then the zeros of $\zeta_{\mathcal L}(s)$ are usually not periodically distributed because the multiplicative group generated by the distinct gap sizes need not be of rank one; see  [\ref{L-vF2}, Chs. 2 and 3].\end{remark}

 \subsection{Exact Tube Formulas for  $p$-Adic  Self-Similar  Strings}\label{exact tf}
   In view of Equation (\ref{sszf}), it follows from the argument given  at the beginning of [\ref{L-vF2}, \S6.4] that every $p$-adic self-similar   string $\mathcal L_p$ is strongly languid, with $\kappa=0$ and $A=r_Ng_K^{-1},$ in the notation of the latter part of Definition \ref{languid}. Indeed, Equation (\ref{sszf}) implies that 
$|\zeta_{\mathcal L_p}(s)|\ll(r_N^{-1}g_K)^{-|\Re(s)|}$, as $\Re(s)\rightarrow -\infty.$ 
Hence, we can apply the distributional  tube formula without error term (i.e., the last part of Theorem 
\ref{dtf} and of Corollary \ref{ftf}) with $W=\mathbb C$. Since by Theorem \ref{lattice}, $\mathcal L_p$ is a lattice string, we obtain (in light of Theorems \ref{rationality}, \ref{periodicity} and   \ref{residue}) the following simpler analogue of Theorem 8.25 in [\ref{L-vF2}], established in [\ref{LapLu3}]:\footnote{We note that instead, we could more generally apply parts (i) and (ii) of Theorem \ref{periodicity} in order to obtain a distributional tube formula with or without error term, valid \emph{without} assuming that all of the complex dimensions of $\mathcal L_p$ are simple. This observation is used in Remark \ref{theta}. } 

\begin{theorem}\label{5.13}
Let $\mathcal{L}_p$ be a $p$-adic self-similar   string with multiplicative generator $r$. Assume that all the complex dimensions of  $\mathcal{L}_p$ are simple. Then, for all $\varepsilon$ with $0<\varepsilon<g_Kr_N^{-1}$, the volume $V_{\mathcal L_p}(\varepsilon)$ is given by the following exact distributional tube formula:
\begin{equation}\label{ssvolume}
V_{\mathcal L_p}(\varepsilon)= \sum_{u=1}^q \varepsilon^{1-\omega_u}G_u(\log_{1/r}\varepsilon^{-1}),
\end{equation}
where $1/r=p^d$ (as in Equation (\ref{Q})), and for each $u=1, \ldots, q,~ G_u$ is a real-valued periodic function of period 1 on $\mathbb R$ corresponding to the line of complex dimensions through 
$\omega_u ~ (\omega_1=D>\Re (\omega_2) \geq \cdots \geq \Re (\omega_q)),$ and is given by the following (conditionally and also distributionally convergent) Fourier series: 

 \begin{equation}\label{Gu}
G_u(x)= \frac{res(\zeta_{\mathcal{L}_p};\omega_u)}{p}
\sum_{n\in \mathbb Z} \frac{e^{2\pi i n x}}{ 1-\omega_u-in\mathbf{p} },
\end{equation}
where \mbox{(}as in Equation (\ref{residue equation}\mbox{)} of Theorem \ref{residue}), 
 \begin{equation*}
res(\zeta_{\mathcal{L}_p}; \omega_u )=
\frac
{\sum_{k =1}^{K} r^{m_{k}'\omega_u}}
{\log{r^{-1}}\sum_{j=1}^N n_j'r^{n_j'\omega_u}}.
\end{equation*}
Moreover, $G_u$ is nonconstant and bounded.
 \end{theorem}

\begin{remark}
In comparing our results with the corresponding results in Chapter 2 and \S8.4 of [\ref{L-vF2}], obtained for real self-similar fractal strings, the reader should keep in mind the following two facts: (i) the simplification brought upon by the ``strong lattice property'' of $p$-adic self-similar   strings; see Theorem \ref{residue} and Remark \ref{stronglattice} above. 
(ii) By construction, any $p$-adic self-similar  string $\mathcal L_p$ (as defined in this section) has total length $L$ equal to one: $L=\mu_H(\mathcal L_p)=\zeta_{\mathcal L_p}(1)=\mu_H(\mathbb Z_p)=1.$ Indeed, for notational simplicity, we have assumed that the similarity transformations 
$\Phi_j$ $(j=1, \dots, N)$ are self-maps of the `unit interval' $\mathbb Z_p$, rather than of an arbitrary `interval' of length $L$ in $\mathbb Q_p$. 
\end{remark}

\begin{remark}[Truncated tube formula with error term]\label{theta}
The   analog of Corollary 8.27 in [\ref{L-vF2}]  holds in the present context, with $2\varepsilon$ replaced by $\varepsilon$  and with $L:=1;$  see the previous remark.
In particular, in light of the method of proof of \emph{loc. cit.,} we have the following `truncated tube formula':
\begin{equation}\label{4.25}
V(\varepsilon)=\varepsilon^{1-D}G(\log_{1/r}\varepsilon^{-1}) + E(\varepsilon), 
\end{equation}
where $G=G_1$ is the nonconstant, bounded periodic function of period 1 given by Equation (\ref{Gu}) of Theorem \ref{5.13} (with $u=1$ and $\omega_1=D$). Here, $E(\varepsilon)$ is an error term that can be estimated much as in \emph{loc. cit.}  In particular,   $E(\varepsilon)=o(1)$ and, moreover, there exists $\delta>0$ such that $\varepsilon^{-(1-D)}E(\varepsilon)=O(\varepsilon^{\delta}),$  as $\varepsilon\rightarrow 0^+.$

Furthermore, since we limit ourselves here to the first line of complex dimensions, and since those complex dimensions are always simple (by part (iii) of Theorem \ref{periodicity}), we do \emph{not}
have to assume (as in Theorem \ref{5.13}) that all the complex dimensions of $\mathcal L_p$ are simple in order for Equation (\ref{4.25}) and the corresponding error estimate for $E(\varepsilon)$ to be valid. (This latter fact can be used to give a direct  proof of the equality  
$D=D_{M}=\sigma$ for any nontrivial $p$-adic self-similar string.) 

More specifically, we note that Equation (\ref{4.25}) and the corresponding error estimate for 
$E(\varepsilon)$ follow from part (i) of Theorem \ref{dtf} (the explicit tube formula with error term, applied to a suitable window), along with the fact that the complex dimensions on the rightmost vertical line $\Re(s)=D$ are simple (according to parts (ii) and (iii) of Theorem \ref{periodicity}, from [\ref{LapLu2}]).
\end{remark}

\begin{remark}
Note that in light of Remark \ref{D<1}, we have $D=\sigma<1$ for any nontrivial $p$-adic self-similar   string $\mathcal L_p$. Hence, we can also apply the distributional tube formula in the general case (when the complex dimensions of $\mathcal L_p$ are not necessarily simple) or, in the present special case of simple complex dimensions (Corollary \ref{ftf}) to obtain a distributional tube formula in this situation, as claimed in Theorem \ref{5.13}. \end{remark}

The next result follows immediately from the truncated tube formula provided in Remark \ref{theta}, along with the corresponding error estimate. 

\begin{theorem}\label{Minkowski nonmeasurable}
A  $p$-adic self-similar string is never Minkowski measurable because it always has multiplicatively periodic oscillations of order $D$ in its geometry. 
\end{theorem}

\begin{example}[Nonarchimedean Fibonacci string and its fractal tube formula] \label{fibonacci}
Let $\Phi_1~ \mbox{and}~  \Phi_2$ be the two affine similarity contraction mappings of $\mathbb {Z}_2$ given by 
\begin{equation*}
\Phi_1(x)=2x \quad \mbox{and} \quad \Phi_2(x)=1+4x,
\end{equation*}
with the respective scaling ratios $r_1=1/2$ and $r_2=1/4$. 
The associated $2$-adic self-similar string (introduced in [\ref{LapLu2}]) with generator $G=3+4\mathbb Z_2$ is called the 
\emph{nonarchimedean} \emph{Fibonacci string} and  denoted by $\mathcal{FS}_2$ (compare with the archimedean counterpart discussed in [\ref{L-vF2}, \S 2.3.2]). 
It is given by the sequence
$\mathcal{FS}_2 =l_1, l_2, l_3, \dots$ and consists (for $m=1,2, \ldots$) of intervals of lengths 
$l_m=2^{-(m+1)}$ with multiplicities $f_{m},$ the Fibonacci numbers.\footnote{These numbers are defined by the recursive formula: $f_{m+1}=f_m+f_{m-1}, f_0=0$ and $f_1=1$.}
Alternatively, the nonarchimedean Fibonacci string is the bounded open subset of $\mathbb Z_2$  given by the following disjoint union of 2-adic intervals (necessarily its 2-adic convex components): 
\[\mathcal{FS}_2=(3+4\mathbb{Z}_2)\cup (6+8\mathbb{Z}_2)\cup (12+16\mathbb{Z}_2)\cup (13+
16\mathbb{Z}_2) \cup \cdots.\]
By Theorem \ref{geometric zeta function}, the geometric zeta function of $\mathcal{FS}_2$ is given (almost exactly as for the archimedean Fibonacci string, cf. \emph{loc. cit.}) by \footnote{The minor difference between the two geometric zeta functions is due to the fact that the real Fibonacci string $\mathcal{FS}$ in [\ref{L-vF2}, \S2.3.2 and Exple. 8.32] has total length 4 whereas the present 2-adic Fibonacci string $\mathcal{FS}_2$
has total length 1.}
 \begin{equation}\label{gzff}
\zeta_{\mathcal{FS}_2}(s)
=\frac{4^{-s}}{1-2^{-s}-4^{-s}}. 
\end{equation}
Hence, the set of complex dimensions  of $\mathcal {FS}_2$ is given by
\begin{equation}\label{cdfs2}
\mathcal{D}_{\mathcal {FS}_2}=\{   D+i n \mathbf{p}~|~n \in \mathbb{Z} \}\cup 
\{  -D+i (n+1/2) \mathbf{p}~|~n \in \mathbb{Z}\}
\end{equation}
with $D=\log_2 \phi$, where $\phi=(1+\sqrt5)/2$ is the golden ratio,  and $\mathbf{p}=2\pi / \log2,$ the oscillatory period of $\mathcal{FS}_2$; see Figure 11.
We refer the interested reader to [\ref{LapLu2}] for additional information concerning the nonarchimedean Fibonacci string. 
\begin{figure}[h]
\psfrag{p}{$\mathbf{p}$}
\psfrag{.5p}{$\frac{1}{2}\mathbf{p}$}
\psfrag{-1}{$-1$}
\psfrag{-D}{$-D$}
\psfrag{0}{$0$}
\psfrag{D}{$D$}
\psfrag{1}{$1$}
\raisebox{-1cm}
{\psfig{figure=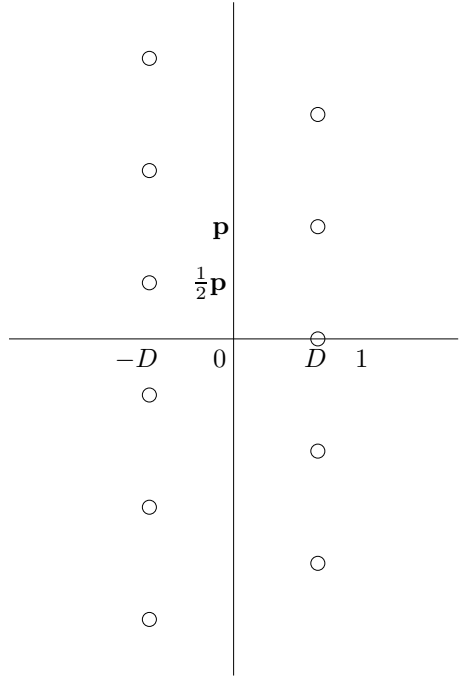, height=9 cm}}
\caption{The complex dimensions of the nonarchimedean Fibonacci string $\mathcal{FS}_2$. Here, $D=\log_2 \phi$ and $\mathbf{p}=2\pi / \log2.$}
\end{figure}

Note that $\zeta_{\mathcal{FS}_2}$ does not have any zero (in the variable $s$) since the equation $4^{-s}=0$ does not have any solution.
Moreover, in agreement with Theorem \ref{rationality}, $\zeta_{\mathcal{FS}_2}$ is a rational function of $z=2^{-s},$ i.e.,  
\begin{equation}\label{zfs2}
\zeta_{\mathcal{FS}_2}(s)=\frac{z^2}{1-z-z^2}. 
\end{equation}
Since, in light of  (\ref{zfs2}), 
 the complex dimensions of $\mathcal{FS}_2$ are simple, we may apply either
  Corollary \ref{ftf} or 
Theorem \ref{5.13}  in order to obtain the following exact fractal tube formula 
for the nonarchimedean Fibonacci string: \footnote{We leave it as an excercise for the interested reader to verify the  computations below or to obtain a direct derivation of $V_{\mathcal {FS}_2}(\varepsilon)$, much as we have done for $V_{\mathcal{E}_p}(\varepsilon)$ and $V_{\mathcal{CS}_3}(\varepsilon)$ in Examples \ref{euler volume} and  \ref{cantor volume}, respectively.}
  \begin{eqnarray*}
V_{\mathcal {FS}_2}(\varepsilon)&=& \frac{1}{2}
 \sum_{\omega\in \mathcal D_{\mathcal{FS}_2}}res(\zeta_{\mathcal {FS}_2};\omega)
  \frac{\varepsilon^{1-\omega}}{1-\omega}\\
   &=&
   \varepsilon^{1-D}G_1(\log_2\varepsilon^{-1})+
      \varepsilon^{1+D-i\mathbf{p}/2}G_2(\log_2\varepsilon^{-1}), 
                          \end{eqnarray*}
                where $G_1$ and $G_2$ are bounded periodic functions of period 1 on $\mathbb R$ given by their respective (conditionally convergent) Fourier series
                \begin{equation}
                G_1(x)= \phi^{-2}\frac{\phi + 2}{10\log 2}
 \sum_{n\in \mathbb Z}
   \frac{e^{2\pi i n x}}{1-D-in\mathbf p} 
                   \end{equation}
                   and 
                    \begin{equation}
                G_2(x)=   \phi^2\frac{3-\phi}{10\log 2}
                 \sum_{n\in \mathbb Z}
   \frac{e^{2\pi i n x}}
  {1+D-i(n+1/2)\mathbf p}.
                     \end{equation}
                     Note that the above Fourier series for $G_1$ and $G_2$ are conditionally (and also distributionally) convergent, for all $x\in\mathbb R.$ Furthermore, the explicit fractal tube formula $V_{\mathcal {FS}_2}(\varepsilon)$ for 
                     $\mathcal{FS}_2$ actually holds pointwise and not just distributionally, as the interested reader may verify via a direct computation (much as in Example \ref{cantor volume} for the 3-adic Cantor string $\mathcal{CS}_3$).
                     \end{example}

 \subsection{The Average Minkowski Content}\label{amc}  
According to Theorem \ref{Minkowski nonmeasurable}, a $p$-adic  self-similar string does not have a well-defined Minkowski content, because it is not Minkowski measurable. Nevertheless, as we shall see in Theorem \ref{average content} below, it does have a suitable `average content' $\mathcal{M}_{av},$ in the following sense:

\begin{definition}\label{defcontent}
Let $\mathcal L_p$ be a $p$-adic fractal string of dimension $D$. The \emph{average Minkowski content}, $\mathcal{M}_{av}$, is defined by the logarithmic Cesaro average
\[
\mathcal{M}_{av}=\mathcal{M}_{av}(\mathcal L_p):=\lim_{T\rightarrow \infty}\frac{1}{\log T}\int_{1/T}^1\varepsilon^{-(1-D)}V_{\mathcal L_p}(\varepsilon) 
\frac{d\varepsilon}{\varepsilon},
\]
provided this limit exists and is a finite positive real number. 
\end{definition}

\begin{theorem}\label{average content} 
Let $\mathcal L_p$ be a $p$-adic self-similar  string of dimension $D$. Then the average Minkowski content of  $\mathcal {L}_p$ exists and is given by the finite positive number
\begin{equation}\label{content}
\mathcal{M}_{av}=\frac{1}{p(1-D)} res(\zeta_{\mathcal L_p}; D)=\frac{1}{p(1-D)} \frac
{\sum_{k =1}^{K} r^{m_{k}'D}}
{\log{r^{-1}}\sum_{j=1}^N n_j'r^{n_j'D}}.
\end{equation} 
\end{theorem}

\begin{remark}
Definition \ref{defcontent} and Theorem \ref{average content} are the exact nonarchimedean counterpart of [\ref{L-vF2}], Definition 8.29 and Theorem 8.30. Furthermore, Theorem \ref{average content} (which is obtained in [\ref{LapLu3}]) follows 
from the `truncated explicit tube formula' given by Equation (\ref{4.25}) in Remark \ref{theta}, along with the corresponding error estimate. 
\end{remark}
\begin{example}[Nonarchimedean Cantor string]\label{average content CS3}
The average Minkowski content of the nonarchimedean Cantor string $\mathcal{CS}_3$ is given by 
\[
\mathcal M_{av}(\mathcal{CS}_3)=\frac{1}{6(\log3-\log2)}.
\]
Indeed, we have seen in Examples \ref{cantor} and \ref{cantor volume}  that $ D=\log_32,  ~res(\zeta_{\mathcal {CS}_3}; D)=1/2\log3$ and $p=3$.
\end{example}

\begin{example}[Nonarchimedean Fibonacci string]\label{amcf}
The average Minkowski content of the nonarchimedean Fibonacci string $\mathcal{FS}_2$ is given by 
\[
\mathcal M_{av}(\mathcal{FS}_2)=\frac{1}{2(\phi +2)(\log2- \log \phi)},
\]
where $\phi=\frac{1+ \sqrt 5}{2}$ is the golden ratio (and so $\phi^{-1}=\frac{\sqrt 5-1}{2}$).  
Indeed, with $p=2$ and   $D=\log_2 \phi $, one can verify that 

\[res(\zeta_{\mathcal {FS}_2}; D)=\frac{1}{(\phi+2)\log2}.\]

Hence, the above expression for $\mathcal M_{av}=\mathcal M_{av}(\mathcal{FS}_2)$ 
follows from Theorem \ref{average content}. Furthermore, note that 
$\log2- \log \phi=\log(\sqrt 5-1).$
Hence, $\mathcal M_{av}$ can be rewritten as follows:
 \[
\mathcal M_{av}=\frac{1}{(5+\sqrt{5})\log(\sqrt 5-1)}.
\]
\end{example}

\begin{remark}
Even though the $p$-adic Euler string $\mathcal E_p$ is not self-similar, we can still use the same methods as those used to prove Theorem \ref{average content}, along with the fractal tube formula (\ref{VolumeEuler}),  in order to calculate its average Minkowski content; we leave the easy verification to the reader. 
Clearly, $\mathcal E_p$ is not Minkowski measurable since, in light of Equation (\ref{cdes}), it has nonreal complex dimension on the real line $\Re(s)=D=0.$
This is also apparent from the fractal tube formula (\ref{VolumeEuler}).
\end{remark}

\section{Concluding Comments}\label{conclusion}  
We close this paper with some comments regarding several possible directions for future research in this area. We hope to address these issues in later work.

\subsection{Ad\`elic Fractal Strings and Their Spectra} \label{adelic}
It would be interesting to unify the archimedean and nonarchimedean settings by appropriately defining \emph{ad\`elic} fractal strings, and then studying the associated spectral zeta functions (as is done for standard archimedean fractal strings in [\ref{Lap1}--\ref{Lap2}] and  [\ref{LapMa},  \ref{LapPo}, \ref{LapPo3}, \ref{L-vF1}, \ref{L-vF2}]). To this aim, the spectrum of these ad\`elic fractal strings should be suitably defined and its study may benefit from Dragovich's work  [\ref{Drag}] on ad\`elic quantum harmonic oscillators. In the process of defining these ad\`elic fractal strings, we expect to make contact with the notion of a fractal membrane (or ``quantized fractal string'') introduced in [\ref{Lap2}, Ch. 3] and rigorously constructed in [\ref{LapNe}] as a Connes-type noncommutative geometric space; see also [\ref{Lap2}, \S 4.2]. The aforementioned spectral zeta function of an ad\`elic fractal string would then be viewed as the (completed) spectral partition function of the associated fractal membrane, in the sense of [\ref{Lap2}]. (See also Remark \ref{harmonic string} and \S\ref{towards}.)

\subsection{Nonarchimedean Fractal Strings in Berkovich Space} \label{berkovich}
    As we have seen in \S3.2, there can only exist \emph{lattice} $p$-adic  self-similar strings, because of the discreteness of the valuation group of $\mathbb Q_p$. However, in the archimedean setting, there are both lattice and nonlattice self-similar strings. We expect that by suitably extending the notion of  
    $p$-adic self-similar  string to Berkovich's $p$-adic analytic space [\ref{Ber}, \ref{Duc}], it can be shown that $p$-adic  self-similar strings are generically nonlattice in this broader setting.  Furthermore, we conjecture that every nonlattice string in the Berkovich projective line  can be approximated by lattice strings with increasingly large oscillatory periods (much as occurs in the archimedean case [\ref{L-vF2}, Ch. 3]).
Finally, we expect that, by constrast with what happens for $p$-adic fractal strings, the volume 
$V_{\mathcal L_p}(\varepsilon)$ will be a continuous function of $\varepsilon$ in this context. (Compare with Remark \ref{real vs. p-adic volume}.)

\subsection{Higher-Dimensional Fractal Tube Formula} \label{higher}
We expect that the higher-dimensional tube formulas obtained by Lapidus and Pearse in [\ref{LapPe1}, \ref{LapPe2}] (as well as, more generally, by those same authors and Winter in [\ref{LPW}]) for archimedean self-similar systems and the associated tilings [\ref{Pe}] in $\mathbb R^d$ have a natural nonarchimedean counterpart in the $d$-dimensional $p$-adic space $\mathbb Q_p^d,$ for any integer $d\geq 1$. In the latter $p$-adic case, the corresponding `tubular zeta function'
 $\zeta_{\mathcal T_p}(\varepsilon; s)$ (when $d=1$, see Remark \ref{tzf}) should have a more complicated expression than in the one-dimensional situation, and should involve both the inner radii and the `curvature' of the generators (see [\ref{LapPe1}--\ref{LPW}] for the archimedean case.) of the tiling (or $p$-adic fractal spray) $\mathcal T_p.$ Moreover, by analogy with what is expected to happen in the Euclidean case  [\ref{LapPe1}--\ref{LPW}], the coefficients of the resulting higher-dimensional tube formula should have an appropriate interpretation in terms of yet to be suitably `nonarchimedean fractal curvatures' associated with each complex and integral dimension of $\mathcal T_p$.  Finally, by analogy with the archimedean case (for $d\geq1$, see [\ref{LapPe1}] and [\ref{LPW}]), the $p$-adic higher-dimensional fractal tube formula should take the same form as in 
 Equation (\ref{sum residue}), except with $\zeta_{\mathcal L_p}(\varepsilon; s)$ given by a different expression from the one in (\ref{tubular zeta}) where $d=1$, and with $\mathcal D_{\mathcal L_p}(W)$ replaced by 
 $ \mathcal D_{\mathcal L_p}(W)\cup \{0, 1, \ldots, d\}$, as well as (for nonarchimedean self-similar tilings) with $W=\mathbb C$ and $\mathcal R_p(\varepsilon)\equiv 0$ in the counterpart of Equation (\ref{tubular zeta}) or (\ref{real tubular zeta}). 
 In the future, we plan to investigate the above problems along with related question pertaining to fractal geometry and geometric measure theory in nonarchimedean spaces. 

\subsection{Towards Nonarchimedean Bergman Spaces and Toeplitz Algebras}\label{towards}
We close this discussion by pointing out a long-term problem that involves challenging and seemingly wide open questions in nonarchimedean harmonic and functional analysis. 

In [\ref{Lap2}], fractal membranes were introduced as suitable quantized analogues of fractal strings. They were viewed heuristically as infinite, ad\`elic noncommutative tori but were also proposed to be properly defined as noncommutative spaces (in the sense of Connes, [\ref{Con}]).

A rigorous construction of archimedean fractal membranes is provided by Lapidus and Nest in 
[\ref{LapNe}]; see also [\ref{Lap2}, \S4.2.1]. It involves, in particular, Toeplitz operators acting on Bergman spaces (in the standard setting of archimedean complex, harmonic and functional analysis). 
The resulting Toeplitz algebra is the $C^*$-algebra $\mathcal T$ which is represented on a suitable Hilbert space $\mathcal H$ (an infinite tensor product of Bergman spaces, one space for each `circle' in the underlying ad\`elic infinite dimensional torus representing the fractal membrane or, equivalently, one space for each interval of the original fractal string being `quantized'). 

Then, the noncommutative space representing the given (archimedean) fractal membrane in a `spectral triple'
\begin{equation}\label{spectral triple}
\mathcal {ST}=\{\mathcal A, \mathcal H, \mathcal D\}, 
\end{equation}  
where $\mathcal D$ is a suitable Dirac-type operator acting on $\mathcal H$.
(Here, $\mathcal D$ is a specific unbounded self-adjoint operator with compact resolvents and bounded commutators with the elements of a suitable dense subalgebra of $\mathcal A$.) We refer to [\ref{Lap2}, \S4.2.1] for an outline of the construction of $\mathcal {ST}$ and to [\ref{LapNe}] for a precise description. 

From our present perspective, the challenge alluded to at the beginning of this subsection consists in obtaining an appropriate nonarchimedean counterpart of the above construction, and thereby of the notion of a fractal membrane.\footnote{It is mentioned in [\ref{Lap2}, \S2.5 and \S5.6] that a suitable $p$-adic, and even ad\`elic, extension of fractal membranes (and their moduli space) could be potentially very useful for the theory.} (Various aspects of this problem are closely connected with the problems discussed in \S \ref{adelic} and \S \ref{berkovich} above.)

At a more modest (but already quite nontrivial level), we must begin by obtaining appropriate nonarchimedean analogs of classical notions in archimedean harmonic and functional analysis, including especially Bergman spaces [\ref{DS}], as well as Toeplitz operators and the associated Toeplitz algebras [\ref{BS}]. 
As we suggested in \S\ref{berkovich} for different, but related reasons, it would likely be helpful in this context to work with nonarchimedean Berkovich-type spaces [\ref{Ber}, \ref{Duc}] rather than with the traditional $p$-adic spaces. 
It does not seem that much information is available on this subject in the literature on nonarchimedean functional analysis and operator theory, 
 but it would be certainly be interesting to investigate aspects of this problem in the future. We invite the interested reader to do so as well. \\


\bibliographystyle{amsplain}

\end{document}